\documentclass{article}
\usepackage{amsmath, amsthm, amssymb, mathtools, leftindex, tikz-cd, verbatim, bbm}
\usepackage{enumitem}
\usepackage[parfill]{parskip}
\usepackage[a4paper, total={7in, 9in}]{geometry}

\usepackage[maxbibnames=99, backend= biber, style = alphabetic]{biblatex}
\usepackage{hyperref}

\newtheorem{lem}{Lemma}[section]

\newtheorem{prop}[lem]{Proposition}
\newtheorem{thm}[lem]{Theorem}

\newtheorem{defi}[lem]{Definition}

\newtheorem{cor}[lem]{Corollary}

\newtheoremstyle{named}{}{}{\itshape}{}{\bfseries}{.}{.5em}{\thmnote{#3}}

\theoremstyle{named}

\newcommand{\D}{\mathcal{D}}
\renewcommand{\S}{\mathcal{S}}

\newcommand{\E}{\mathcal{E}}

\newcommand{\Z}{\mathbb{Z}}

\newcommand{\A}{\mathbb{A}}
\newcommand{\Q}{\mathbb{Q}}
\newcommand{\CC}{\mathbb{C}}
\newcommand{\Hyp}{\mathbb{H}}
\newcommand{\R}{\mathbb{R}}

\newcommand{\M}{\mathfrak{M}}
\newcommand{\J}{\mathfrak{J}}

\newcommand{\N}{\mathcal{N}}
\renewcommand{\O}{\mathcal{O}}
\renewcommand{\o}{\mathfrak{o}}

\renewcommand{\P}{\mathfrak{P}}
\newcommand{\p}{\mathfrak{p}}
\newcommand{\q}{\mathfrak{q}}

\newcommand{\Pl}{\mathcal{V}}
\newcommand{\Primes}{\mathcal{V}_{\text{fin}}}

\DeclareMathOperator{\Gal}{Gal}

\DeclareMathOperator{\SL}{SL}
\DeclareMathOperator{\GL}{GL}

\DeclareMathOperator{\disc}{disc}
\DeclareMathOperator{\vol}{Vol}

\DeclareMathOperator{\reg}{reg}

\DeclareMathOperator{\Tr}{Tr}
\DeclareMathOperator{\modulo}{mod}

\DeclareMathOperator{\Ram}{Ram}

\DeclareMathOperator{\Emb}{Emb}
\DeclareMathOperator{\nrd}{nrd}
\DeclareMathOperator{\diag}{diag}

\DeclareMathOperator{\trd}{trd}

\begin{document}
\title{Chebotarev geodesic theorem: non-split case.}
\author{Alberto Acosta Reche \thanks{This work was supported by the Engineering and Physical Sciences Research Council [EP/S021590/1], via the
EPSRC Centre for Doctoral Training in Geometry and Number Theory (The London School of Geometry and Number
Theory), University College London.}}
\date{}
\maketitle

\begin{abstract}
We study the prime geodesic theorem (PGT) for congruence subgroups of indefinite quaternion orders. We prove that the geodesic analogue of the Chebotarev density theorem for these groups holds with exponent $25/36 + \varepsilon$. In particular, we deduce that the PGT holds with exponent $25/36 + \varepsilon$ for any congruence subgroup of any indefinite quaternion order over $\Q$. The idea of the proof is to reduce the problem to the split case, which has been handled previously by the author. 
\end{abstract}

\section{Introduction.}
\subsection{Background.}
Let $\{\pm I\} \subset \Gamma \subset \SL_2(\R)$ be a discrete cofinite subgroup. Let $\{P\}_{\Gamma}$ range over hyperbolic $\Gamma$-conjugacy classes and let $X \geq 1$ be a parameter, to be taken large. Let $C_\Gamma(P) = \{\gamma \in \Gamma : \gamma^{-1}P\gamma = P\}$ be the centralizer of $P$ in $\Gamma$. We let $\nu(P) := [\{\pm I\}C_\Gamma(P): \{\pm I\}\langle P\rangle]$. The \emph{prime geodesic theorem} (PGT) is concerned with the asymptotics of the sum
\begin{equation}\label{eq:defiofprimegeodesicsum}
    \Psi_{\Gamma}(X) := \sum_{\substack{\{P\}_{\Gamma}\\N(P)\leq X\\
    \Tr(P) > 2}}\frac{\log N(P)}{\nu(P)}
\end{equation}
as $X \to \infty$. The sum can be interpreted as a weighted sum of lengths of closed geodesics of the hyperbolic orbifold $\Gamma\backslash \SL_2(\R)$, hence the use of the adjective \emph{geodesic}. In analogy with the prime number theorem we have 
\begin{equation}\label{eq:weakprimegeodesic}
\Psi_\Gamma(X) \sim X \quad \text{ as }\quad X \to \infty.
\end{equation}
When $\Gamma\backslash \Hyp$ is compact this result is originally due to Huber, see \cite[Satz 9]{huberI}. For a general cofinite $\Gamma$ one can show directly from the Selberg trace formula that
\begin{equation}\label{eq:basicprimegeodesic}
\Psi_\Gamma(X) = X + \S_\Gamma(X) + O(X^{3/4}),
\end{equation}
where $\S_\Gamma(X)$ is a secondary term given explicitly in terms of the small eigenvalues of $\Gamma\backslash \Hyp$. With a slightly bigger error term, this result is due to Sarnak in \cite[Theorem 3.16]{sarnakthesis} and Hejhal in \cite[Chapter 10, Theorem 3.4]{hejhal2}. For a particularly clear exposition of the proof, see \cite[Theorem 10.5]{iwaniec}.

The error term in \eqref{eq:basicprimegeodesic} can be improved in some situations, in particular when $\Gamma = \SL_2(\Z)$. In this case there are no small eigenvaulues, so $\S_\Gamma(X) = 0$. The PGT for the modular surface is particularly appealing to number theorists because of its connection with class numbers. Indeed, Sarnak showed \cite{sarnak82} that 
\begin{equation}\label{eq:connectiontoclassnumbers}
    \Psi_{\SL_2(\Z)}(X) = 2 \sum_{2 < t \leq X^{1/2} + X^{-1/2}}\sum_{\substack{d, u\\du^2 = t^2 - 4}} h(d) \log(\varepsilon_d) = 2\sum_{2 < t \leq X^{1/2} + X^{-1/2}} \sum_{\substack{d, u\\ du^2 = t^2 - 4}}  \sqrt{d}L(1, \chi_d).
\end{equation}
In this sum $h(d)$ and $\log \varepsilon_d$ are the class number and the regulator, respectively, of the unique quadratic order of discriminant $d$, $\chi_d$ is the quadratic character (not necessarily primitive) associated to the discriminant $d$, and $L(s, \chi_d)$ is the corresponding Dirichlet series. The second equality follows from the class number formula. 

The PGT for $\SL_2(\Z)$ was studied by Iwaniec \cite{Iwaniec1984}, Luo and Sarnak \cite{luosarnak}, and Cai \cite{Cai2002}, among other mathematicians.  At the moment of writing this paper, the best estimate for the error term in \eqref{eq:basicprimegeodesic} is
\begin{equation}\label{eq:strongprimegeodesic}
\Psi_\Gamma(X)  = X + O_\varepsilon(X^{25/36 + \varepsilon}),
\end{equation}
due to Soundararajan and Young \cite{Soundararajan_2013}. In \cite{reche2026chebotarevgeodesictheoremsplit} we have showed that the same estimate holds for any congruence subgroup of $\SL_2(\Z)$. More generally, one can study the twisted sums 
\begin{equation}\label{eq:defitwistedchebyshevsum}
    \Psi_{\SL_2(\Z)}(X; f) := \sum_{\substack{\{P\}_{\SL_2(\Z)}\\N(P)\leq X\\
    \Tr(P) > 2}}\frac{\log N(P)}{\nu(P)}f(P \modulo q),
\end{equation}
where $q$ is a positive integer and $f$ is a class function on $\SL_2(\Z/q\Z)$. Sarnak explained in \cite[Section 3]{sarnakthesis} how the asymptotic study of these sums can be seen as the geodesic analogue of the Chebotarev density theorem. From \cite[Chapter 10, Theorem 3.4]{hejhal2} it follows that 
\begin{equation}\label{eq:weakchebotarevgeodesic}
    \Psi_{\SL_2(\Z)}(X; f) =  \langle f, \mathbbm{1}\rangle_{\SL_2(\Z/q\Z)} X + O_{f, \theta, \varepsilon}(X^{3/4}(\log X)^{1/2}).
\end{equation}
In this formula we use the notation
\begin{equation*}
    \langle f_1, f_2 \rangle_G  := \frac{1}{|G|}\sum_{g \in G}f_1(g)\overline{f_2(g)}
\end{equation*}
for the inner product of two functions $f_1, f_2$ on a finite group $G$. We call a result like \eqref{eq:weakchebotarevgeodesic} the \emph{Chebotarev geodesic theorem} for $\SL_2(\Z)$ (CGT). In \cite{reche2026chebotarevgeodesictheoremsplit} we have given a new estimate for the error term in the CGT for $\SL_2(\Z)$. Precisely, we have shown that  
\begin{equation}\label{eq:strongchebotarevsplit}
    \Psi_{\SL_2(\Z)}(X; f) = \langle f, \mathbbm{1}\rangle_{\SL_2(\Z/q\Z)} X + O_{f, \theta, \varepsilon}(X^{2/3 + \theta/6 +\varepsilon})
\end{equation} 
holds for $X \geq 1$, where $\theta > 0$ is an \emph{admissible exponent} whose definition we recall below.

\begin{defi}\label{defi:admissibleexponent}
We say that $\theta > 0$ is \emph{admissible} if, for any Dirichlet character $\chi_1$, there exists $A = A(\chi_1) > 0$ such that
\begin{equation*}
    |L(s; \chi_1 \chi_2)| \ll_{\chi_1, \theta} \q(\chi_2)^\theta |s|^A
\end{equation*}
holds for $\text{Re}(s) = 1/2$ and any real Dirichlet character $\chi_2$.
\end{defi}

By the work of Petrow and Young we know that the exponent $\theta = 1/6 + \varepsilon$ is admissible for any $\varepsilon > 0$, see \cite{petrowyoung20} and \cite{petrowyoung}. Thus, for any $\varepsilon > 0$ we have the unconditional bound 
\begin{equation}\label{eq:strongsplitchebotarev}
\Psi_{\SL_2(\Z)}(X; f) = \langle f, \mathbbm{1}\rangle_{\SL_2(\Z/q\Z)} X + O_{f, \varepsilon}(X^{25/36 + \varepsilon}).
\end{equation}

\subsection{Results.}
The goal of this paper is to extend the estimate \eqref{eq:strongchebotarevsplit} to indefinite quaternion algebras. We let $\D$ denote an indefinite division quaternion algebra over $\Q$, and let $\O$ be a maximal order in $\D$. Recall that the discriminant $\disc(\D)$ is the product of the primes $p$ for which $\D\otimes_p \Q_p$ is a division algebra. Since we assume $\D$ to be indefinite, we have an isomorphism $\D \otimes_\Q \R \simeq M_2(\R)$. Via this isomorphism the group $\O^1$ of unimodular units of $\O$ becomes a cocompact discrete subgroup of $\SL_2(\R)$. For a positive integer $q$, we let $\O^1(q):= (1 + q\O) \cap \O^1$ denote the principal congruence subgroup of $\O^1$ of level $q$. We consider an arbitrary congruence subgroup $\O^1(q) \subset \Gamma \subset \O^1$ of level $q$ and let $\varphi$ be a class function on $\widetilde{\Gamma} := \O^1(q)\backslash \Gamma$. In analogy with the twisted Chebyshev sums \eqref{eq:defitwistedchebyshevsum} we consider 
\begin{equation}\label{eq:chebyshevquaternion}
    \Psi_{\Gamma}(X; \varphi) := [\{\pm I\}\Gamma: \Gamma]\sum_{\substack{\{P\}_\Gamma\\
    N(P) \leq X\\
    \Tr(P) > 2}} \frac{\log N(P)}{\nu(P)}\varphi(P \modulo q).
\end{equation}

We are now ready to state our main result.

\begin{thm}\label{thm:classicalmatching}
Let $\Gamma$ be a congruence subgroup of $\O^1$ of level $q$, and let $\varphi$ be a class function on $\widetilde{\Gamma}$. Write $q' := \disc(\D) q$. Then, there exists a class function $f$ on $\SL_2(\Z/q'\Z)$ such that, for any $X \geq 1$ and any $\varepsilon > 0$, we have
\begin{equation*}
    \Psi_{\Gamma}(X; \varphi) = \Psi_{\SL_2(\Z)}(X; f) + O_{\varepsilon, \varphi}(X^{1/2 + \varepsilon}).
\end{equation*}
\end{thm}

When this result is combined with \eqref{eq:strongchebotarevsplit} we obtain a strong form of the CGT for quaternion congruence subgroups.

\begin{cor}\label{cor:nonsplitchebotarev}
Let the notation be as in Theorem \ref{thm:classicalmatching}. Then, for any admissible exponent $\theta > 0$ and any $\varepsilon > 0$, we have
\begin{equation*}
    \Psi_{\Gamma}(X; \varphi) = \langle \varphi, \mathbbm{1}\rangle_{\widetilde{\Gamma}} X + O_{\varphi, \theta, \varepsilon}(X^{2/3 + \theta/6 + \varepsilon}).
\end{equation*}
\end{cor}
In particular, considering the constant function $\varphi = \mathbbm{1}$ and using the admissible exponent $\theta = 1/6 + \varepsilon$ we obtain unconditionally a strong form of the PGT for quaternion congruence subgroups.  

\begin{cor}\label{cor:primegeodesicquaternion}
Let $\Gamma$ be a congruence subgroup of $\O^1$. Then, for any $\varepsilon > 0$ we have
\begin{equation*}
    [\{\pm I\} \Gamma: \Gamma]\sum_{\substack{\{P\}_{\Gamma}\\N(P)\leq X\\
    \Tr(P) > 2}}\frac{\log N(P)}{\nu(P)} = X + O_{\Gamma, \varepsilon}(X^{25/36 + \varepsilon}).
\end{equation*}
\end{cor}

Thus, we are able to extend the current strongest form of the PGT for $\SL_2(\Z)$ to a much wider class of groups. Corollary \ref{cor:primegeodesicquaternion} should be compared with \cite[Theorem 1.1]{tang2025}, which applies only to principal congruence subgroups of $\O^1$ and has exponent $7/10 + \varepsilon$ in the error term.

\subsection{Summary of the proof of Theorem \ref{thm:classicalmatching}.}
As in the proof of \cite[Theorem 1.1]{tang2025}, the main idea
in our proof of Theorem \ref{thm:classicalmatching} is to translate the problem into one of matching certain orbital integrals. Before going any further, we will introduce some notation. Let $B$ be an indefinite quaternion algebra over $\Q$, and let $\O$ be a maximal order in $B$. We let $\widehat{B}^\times$ and $\widehat{B}^1$ denote the groups of finite adeles over $\Q$ corresponding to $B^\times$ and $B^1$, respectively. The groups $\widehat{B}^\times$ and $\widehat{B}^1$ have maximal compact subgroups given by $\widehat{\O}^\times$ and $\widehat{\O}^1$, respectively. For each positive integer $q$, we let $\widehat{\O}^\times(q)$ and $\widehat{\O}^1(q)$ denote the principal congruence subgroups of level $q$ of $\widehat{\O}^\times$ and $\widehat{\O}^1$, respectively. 

By \cite[Lemma 2.1]{reche2026chebotarevgeodesictheoremsplit}, we can assume that $\Gamma = \O^1$ in  
Theorem \ref{thm:classicalmatching}. By strong approximation for $\O^1$ we can identify $(\O/q\O)^1$ with $\widehat{\O}^1(q) \backslash \widehat{\O}^1$. Thus, there is a natural bijection between functions $\varphi$ on $(\O/q\O)^1$ and functions $\widehat{\varphi}$ on $\widehat{\O}^1$ of level $q$, that is, functions which are invariant under translation by elements of $\widehat{\O}^1(q)$.
\begin{defi}\label{defi:stablefunction}
We say that a function $\varphi$ on $(\O/q\O)^1$ is \emph{stable} if it is invariant under conjugation by $(\O/q\O)^\times$. 
\end{defi}
Equivalently, $\varphi$ is stable if and only if $\widehat{\varphi}$ is invariant under conjugation by $\widehat{\O}^\times$. Under the assumption that $\varphi$ is stable we have the following generalization of identity \eqref{eq:connectiontoclassnumbers} 
\begin{equation}\label{eq:approxformula}
    \Psi_{\O^1}(X; \varphi) = \sum_{2 < t \leq X^{1/2} + X^{-1/2}}     2\reg(E_t) h(E_t) \vol_{\widehat{E_t}^\times}(\widehat{\o_{E_t}}^\times)\O_{\widehat{B}^\times}(\widehat{\varphi}; E_t, x_t),
\end{equation}
see Proposition \ref{prop:expressionchebyshevprimegeodesic} below. In this equation $E_t$ is the quadratic field $E_t := \Q[X]/(X^2 - tX + 1)$, its ring of integers is denoted by $\o_t$, its regulator by $\reg(E_t)$, its class number by $h(E_t)$, the element $x_t \in \o_t$ is defined as $x_t := X \modulo (X^2 - tX + 1)$, and $\O_{\widehat{B}^\times}(\widehat{\varphi}; E_t, x_t)$ is an orbital integral over finite adeles defined in equation \eqref{eq:defifinitepartorbitalintegral}.

Suppose now that $B = \D$ is an indefinite quaternion division algebra over $\Q$. From Proposition \ref{prop:explicitmatching} it follows that, given a function $\widehat{\varphi}$ on $\widehat{\O}^1$ of level $q$, we can find a stable function $\widehat{f}$ on $\SL_2(\widehat{\Z})$ of level $q' :=\disc(\D) q$ such that 
\begin{equation}\label{eq:matchingoforbitalintegralsintro}
    \O_{\widehat{\D}^\times}(\widehat{\varphi}; E_t, x_t) = \O_{\GL_2(\widehat{\Q})}(\widehat{f}; E_t, x_t)
\end{equation}
for any $t > 2$, see Section \ref{subsec:proofofclassicalmatching} for more details. Thus, if we use \eqref{eq:approxformula} for $B = \D$, then use \eqref{eq:matchingoforbitalintegralsintro}, and finally use \eqref{eq:approxformula} again, but this time for $B = M_2(\Q)$, we obtain the following result.

\begin{prop}\label{prop:classicalmatchingforstable}
Let $\D$ be an indefinite quaternion division algebra over $\Q$, and let $\O$ denote a maximal order in $\D$. Let $\varphi$ be a stable class function on $(\O/q\O)^1$, and write $q' := \disc(\D) q$. Then, there exists a stable class function $f$ on $\SL_2(\Z/q'\Z)$ such that, for any $X \geq 1$, we have
\begin{equation*}
    \Psi_{\O^1}(X; \varphi) = \Psi_{\SL_2(\Z)}(X; f).
\end{equation*}
\end{prop}
The details of the proof of Proposition \ref{prop:classicalmatchingforstable} are given in Section \ref{subsec:proofofclassicalmatching}. One question still remains: how to remove the assumption that $\varphi$ is stable in the previous arguments? This question is answered by the following result, which will be proved in Section \ref{subsec:proofofalmoststability}.
\begin{prop}\label{prop:almoststabilityclassical}
Let $\varphi$ be a class function on $(\O/q\O)^1$ and let $\gamma \in (\O/q\O)^\times$. Define $\varphi^\gamma$ by the formula $\varphi^\gamma(x) := \varphi(\gamma x \gamma^{-1})$. Then, for any $\varepsilon > 0$, we have 
\begin{equation*}
    \Psi_{\O^1}(X; \varphi) = \Psi_{\O^1}(X; \varphi^\gamma) + O_{\varepsilon, \varphi}(X^{1/2 + \varepsilon}).
\end{equation*}
\end{prop}
Thus, in the proof of Theorem \ref{thm:classicalmatching}, after averaging over $(\O/q\O)^\times$ we can assume that $\varphi$ is stable, at the cost of a small error term. Applying Proposition \ref{prop:classicalmatchingforstable} to the resulting stable function completes the proof of Theorem \ref{thm:classicalmatching}.

\subsection{Remarks on explicit matching of orbital integrals.}
The matching of orbital integrals from $\D^\times$ to $\GL_2$ is an important step in the proof of the Jacquet--Langlands correspondence, see for example \cite[Section 8]{corvallisgelbartjacquet}. Even though several versions of this matching are available in the literature, we could not find a version that suited our purposes. Our version of the matching has the additional feature that it keeps track of the support and the level of $\widehat{f}$ as a function of $\widehat{\varphi}$, as opposed to previous versions such as \cite[Lemma 8.10]{corvallisgelbartjacquet}, which do not retain this information.


In \cite{tang2025}, the authors proved Proposition \ref{prop:classicalmatchingforstable} for the particular case when $\varphi$ is the characteristic function of a principal congruence subgroup. Our Proposition \ref{prop:mymatching} below complements their work by treating different test functions. Combining our work with that of \cite{tang2025} leads to a completely explicit version of the matching of local orbital integrals in the case of inner forms of $\GL_2$. 

Proposition \ref{prop:almoststabilityclassical} is of a very different nature than Proposition \ref{prop:classicalmatchingforstable}. It is closely related to the stabilization of the elliptic contribution to the Arthur trace formula for $B^1$, see \cite{Labesse_Langlands_1979}. However, our proof of Proposition \ref{prop:almoststabilityclassical} in Section \ref{subsec:proofofalmoststability} does not appeal to any background results on the stable trace formula. Instead, we proceed ``by hand" using only Artin reciprocity for quadratic field extensions. 


\subsection{Notation.}
\begin{itemize}
    \item Let $G$ be a group and $\varphi: G \rightarrow Y$ a function. For $g, x \in G$ we define $\varphi^g(x) := \varphi(g^{-1}xg)$.  
    \item Let $G$ be a group and $\iota: X \rightarrow G$ a function. For $g \in G$ and $x \in X$ we define $\prescript{g}{}\iota(x):= g\iota(x)g^{-1}$. 
    \item Let $A, B, C$ be subsets (not necessarily subgroups) of an abelian group $X$ with addition map $f: X \times X \rightarrow X$. We write $C = A \oplus B$ when $f|_{A \times B}$ is injective and $C = f(A \times B)$.
\end{itemize}

\subsection{Acknowledgements.}
The author would like to express his gratitude to his advisor Ian Petrow for helpful discussions and for carefully reading the present work. The author would also like to thank Yiannis Petridis for suggesting the problem to the author and for his guidance in the initial stages of the project.

\section{Local orbital integrals.}\label{sec:localarguments}
\subsection{Background and statements.}
Much of the notation in this section is taken from \cite{tang2025}. For a very thorough reference for quaternion algebras and their arithmetic we refer to Voight's book \cite{voight}.

In this section $F$ is a nonarchimedean local field of characteristic zero, $\o$ is its ring of integers, $\p$ is the maximal ideal of $F$, $k(\p) := \o/\p$ is the residue field of $\o$, we let $\varpi \in \p$ be a uniformizer of $F$ and $v: F \rightarrow \Z \cup \{\infty\}$ denotes the valuation of $F$ normalized so that $v(\varpi) = 1$, with the convention that $v(0) = \infty$. Furthermore, we let $q := \# k(\p)$, so that $k(\p) \simeq \mathbb{F}_q$. Whenever we consider an order in an $F$-algebra, it will be an $\o$-order.

The letter $E$ will denote a separable $F$-algebra of degree two. This means that, either $E \simeq F \times F$, or else $E/F$ is a field extension of degree two. We let $\o_E$ denote the maximal order of $E$, and let $\p_E$ denote the Jacobson ideal of $\o_E$. If $E \simeq F \times F$, then $\o_E = \o \oplus \o$ and $\p_E = \p \oplus \p$. Otherwise, if $E$ is a field, then $\o_E$ is the ring of integers of $E$ and $\p_E$ is its maximal ideal.

We let $\trd_E, \nrd_E : E \rightarrow F$ denote the usual trace and norm maps, respectively. We know that the characteristic equation
\begin{equation*}
    x^2 - \trd_E(x)x + \nrd_E(x) = 0
\end{equation*}
holds for all $x \in E$. The trace map $\trd_E$ is $F$-linear, while $\nrd_E$ is multiplicative.

Recall there is only one isomorphism class of division quaternion algebras over $F$ \cite[Theorem 13.3.11]{voight}. We let $\D$ denote a representative of this isomorphism class. The letter $B$ will denote a quaternion algebra over $F$, so that either $B \simeq M_2(F)$ or else $B \simeq \D$. We equip $B^\times$ with a nontrivial Haar measure. Note that $B^\times$ is unimodular, so that the notions of right-invariant and left-invariant Haar measure agree. 

We let $\Emb(E; B)$ denote the set of $F$-algebra embeddings of $E$ into $B$. The quaternion algebra $B$ comes equipped with maps $\trd_B, \nrd_B: B\rightarrow F$ called \emph{reduced trace} and \emph{reduced norm}, respectively. They are characterized by the following property: for any embedding $\iota \in \Emb(E; B)$ and any $x \in E$ we have $\trd_B(\iota(x)) = \trd_E(x)$ and $\nrd_B(\iota(x)) = \nrd_E(x)$. When there is no risk of confusion we will drop the subscripts and simply write $\trd$ and $\nrd$. We deduce that the characteristic equation 
\begin{equation*}
    x^2 - \trd_B(x)x + \nrd_B(x) = 0
\end{equation*}
holds for all $x \in B$. Both maps $\trd_B$ and $\nrd_B$ are surjective. The map $\trd_B$ is $F$-linear, while $\nrd_B$ is multiplicative, and we have $B^\times = \nrd_B^{-1}(F^\times)$. In the case when $B = M_2(F)$ then $\trd_B$ and $\nrd_B$ are just the trace and the determinant maps, respectively.

Recall that each $E$ and each $B$ comes equipped with a \emph{standard involution}, denoted by $x\mapsto \overline{x}$. It is given explicitly by $\overline{x} = \trd(x) - x$, and we have $\overline{xy} = \overline{y}\overline{x}$ as well as $\nrd(x) = x\overline{x}$. Note that if $\iota \in \Emb(E; B)$ and $x \in E$, then $\overline{\iota(x)} = \iota(\overline{x})$, so there is no risk of confusion with respect to the choice of ambient algebra where $x$ lives. For example, when $E$ is a field, then $x \mapsto \overline{x}$ is the nontrivial Galois automorphism of the extension $E/F$. 

We give a name to the elements of $B$ which generate separable $F$-algebras of degree two.
\begin{defi}\label{defi:regularelements}
We say that $x \in B$ is \emph{regular} if $F[x]\subset B$ is a separable $F$-algebra of degree two. We let $B_{\reg}$ denote the set of regular elements of $B$.
\end{defi}
Clearly, $B_{\reg} = \{x \in B: \trd(x)^2 - 4\nrd(x) \neq 0\}$. For example, when $B = \D$ we have $\D_{\reg} = \D \smallsetminus F$. 

Suppose we are given an embedding $\iota \in \Emb(E; B)$. Thanks to this embedding we can translate a Haar measure from $E^\times$ to $\iota(E^\times)$. Note that $\iota(E^\times)$ and $B^\times$ are both unimodular. We equip $\iota(E^\times)\backslash B^\times$ with the unique measure which is compatible with the Haar measures on $\iota(E^\times)$ and $B^\times$ and which is invariant under right translation by elements of $B^\times$ \cite[Ch. VII, §2, No. 6, Cor. 2]{bourbakiintegrationii}.

We say that a function $f$ on $B^\times$ is \emph{admissible} if, for any $E$ as above, any $\iota \in \Emb(E; B)$ and any $x \in E^\times \smallsetminus F^\times$ the function $g \mapsto f(g^{-1}\iota(x)g)$ lies in $L^1(\iota(E)^\times\backslash B^\times)$. For example, if $f \in C_c^\infty(B^\times)$ or $f \in C_c^\infty(B^1)$, then $f$ is admissible, and these are the only two examples that we will be concerned with. For an admissible function $f$ we can define the orbital integral
\begin{equation}\label{eq:defioforbitalintegral}
    \O_{B^\times}(f; E, x) := \int_{\iota(E^\times)\backslash B^\times}f(g^{-1}\iota(x)g)\, dg.
\end{equation} 
If $\Emb(E; B) = \varnothing$, then we agree by convention that $\O_{B^\times}(f; E, x) = 0$. We know that $\Emb(E; B) = \varnothing$ occurs precisely when $E \simeq F \times F$ and $B \simeq \D$, see Lemma \ref{lem:optimalembeddingsKO}. By the Noether-Skolem theorem \cite[Theorem 7.7.1]{voight} any two embeddings $\iota_1, \iota_2 \in \Emb(E; B)$ are conjugate by $B^\times$, so the integral in \eqref{eq:defioforbitalintegral} is independent of the choice of embedding $\iota$. 

If $a \in F$, we let $[+a]f$ denote the function defined by the identity $[+a]f(x) = f(x- a)$. Similarly, if $b \in F^\times$ we let $[\times b]f$ be defined by the identity $[\times b]f(x) = f(b^{-1}x)$. In particular, if $A \subset B$ is any subset we see that
\begin{equation}\label{eq:translationdilationsets}
    [+a]\mathbbm{1}_{A} = \mathbbm{1}_{a + A},\quad [\times b]\mathbbm{1}_{A} = \mathbbm{1}_{bA}.
\end{equation}
These definitions are compatible with the orbital integrals, since we have
\begin{equation}\label{eq:transformationorbital}
    \O_{B^\times}([+a]f; E, x + a) = \O_{B^\times}(f; E, x), \quad \O_{B^\times}([\times b]f; E, bx) = \O_{B^\times}(f; E, x).
\end{equation}

In this section we want to compare orbital integrals for the two isomorphism classes of quaternion algebras.

\begin{defi}\label{defi:matchingorbital}
Let $\varphi$ and $f$ be admissible functions on $\D^\times$ and $\GL_2(F)$, respectively. We say that $\varphi$ and $f$ are \emph{matching functions}, and denote this situation by $\varphi \leftrightarrow f$, if we have $\O_{\D^\times}(\varphi; E, x) = \O_{\GL_2(F)}(f; E, x)$ for any $E$ and any $x \in E^\times \smallsetminus F$. 
\end{defi}
Recall that we have agreed that the orbital integral vanishes if there are no embeddings of $E$ in $B$. In more explicit terms, $\varphi \leftrightarrow f$ means the following:
\begin{itemize}
    \item When $E \simeq F \times F$ and $x \in E^\times \smallsetminus F$ we have $\O_{\GL_2(F)}(f; E, x) = 0$.
    \item For any quadratic field extension $E/F$ and any element $x \in E^\times \smallsetminus F$ we have 
\begin{equation*}
    \O_{\D^\times}(\varphi; E, x) = \O_{\GL_2(F)}(f; E, x),
\end{equation*}
\end{itemize}

We record a very simple consequence of the definition of matching functions: 
\begin{equation}\label{eq:matchingonBtimesimpliesB1}
(\varphi \leftrightarrow f) \implies (\varphi \mathbbm{1}_{\D^1} \leftrightarrow f\mathbbm{1}_{\SL_2(F)}).
\end{equation}
Indeed, since conjugation by $B^\times$ preserves the reduced norm, we have 
\begin{equation*}
    \O_{\D^\times}(\varphi \mathbbm{1}_{\D^1}; E, x) = \begin{dcases}
    \O_{\D^\times}(\varphi; E, x), & \text{ if }\nrd(x) = 1,\\
    0, & \text{ otherwise},
    \end{dcases}
\end{equation*}
and similarly for the $\GL_2(F)$-orbital integrals.

Recall that $\D$ has a unique maximal order \cite[Proposition 13.3.4]{voight}, which we denote by $\O_\D$. On the other hand, $M_2(F)$ has many maximal orders, but they are all conjugate to $\M := M_2(\o)$. When we refer to a maximal order $\O$ in $M_2(F)$, we will always understand that $\O = \M$. 
\begin{defi}\label{defi:subgroupsdefininglevel}
When $\O$ is a maximal order in $B$ and $n \geq 0$ we define groups $\O^\times(\p^n)$ and $\O^1(\p^n)$ as follows:
\begin{itemize}
    \item If $n = 0$, then $\O^\times(\p^0) := \O^\times$ and $\O^1(\p^0) := \O^1$.
    \item If $n \geq 1$, then $\O^\times(\p^n) := 1 + \p^n \O$ and $\O^1(\p^n) := \O^\times(\p^n) \cap \O^1$.
\end{itemize} 
\end{defi}

\begin{defi}\label{defi:levelfunction}
Given a function $f$ on $B^\times$, we say that $f$ is of level $\p^n$ if $f(gu) = f(g) = f(ug)$ for all $g \in B^\times$ and $u \in \O^\times(\p^n)$. Similarly, given a function $f$ on $B^1$, we say that $f$ is of level $\p^n$ if $f(gu) = f(g) = f(ug)$ for all $g \in B^1$ and $u \in \O^1(\p^n)$. 
\end{defi}
The computations in the remainder of this section can be summarized in the following result.

\begin{prop}\label{prop:explicitmatching}
Every $\varphi \in C_c^\infty(\D^\times)$ of level $\p^n$ has a matching function $f \in C_c^\infty(\GL_2(F))$ of level $\p^{n+1}$ that is invariant under conjugation by $\GL_2(\o)$. Furthermore, if $\varphi$ is supported on $\O_\D^\times$, then $f$ can be chosen with the extra property that it is supported on $\GL_2(\o)$.

Similarly, every $\varphi \in C_c^\infty(\D^1)$ of level $\p^n$ has a matching function $f \in C_c^\infty(\SL_2(F))$ of level $\p^{n+1}$ that is invariant under conjugation by $\GL_2(\o)$. Furthermore, if $\varphi$ is supported on $\O_\D^1$, then $f$ can be chosen with the extra property that it is supported on $\SL_2(\o)$.
\end{prop}

We will now explain in detail how $f$ is constructed in terms of $\varphi$.

Let $E$ be a separable $F$-algebra of degree two. We recall that the orders in $E$ are indexed by $\mathbb{Z}_{\geq 0}$. Indeed, any order of $E$ is of the form $\o_{E, n} := \o + \varpi^n \o_E$ for a unique $n\geq 0$. 

If $\O$ is an order in a quaternion algebra $B$ and $\iota \in \Emb(E; B)$ is an embedding, then $\iota^{-1}(\O)$ is an order in $E$, which must therefore equal $\o_{E, n}$ for some $n\geq 0$. This justifies the following definition.
\begin{defi}
Given $\iota \in \Emb(E; B)$ and an order $\O$ in $B$, we let $n_\O(\iota)$ be the nonnegative integer such that $\iota^{-1}(\O) = \o_{E, n_\O(\iota)}$. If $x \in B_{\reg}$ we let $n_\O(x) := n_\O(\iota)$ for the inclusion $\iota: F[x] \subset B$.
\end{defi}
Given an order $\O$ in a quaternion algebra, not necessarily maximal, we let $\P_\O$ be its Jacobson radical, and we let $K_\O$ be its normalizer inside $B^\times$, so that $K_\O := \{g \in B^\times : g^{-1}\O g= O\}$. We use the notation $\P_\D$ and $K_\D$ to refer to the Jacobson radical and normalizer, respectively, of the unique maximal order of the division algebra $\D$. On top of the maximal orders $\O_\D$ and $\M$ we will also need to consider the standard non-maximal hereditary order of $M_2(F)$, denoted by $\J$. We quickly recall a bit of background about these three orders. 
\begin{itemize}
    \item When $B = M_2(F)$ every maximal order is conjugate to $\M = M_2(\o)$. The normalizer of $\M$ in $\GL_2(F)$ is given by $K_\M = \varpi^\Z \M^\times$. The Jacobson radical of $\M$ is $\P_\M = \varpi \M$. Note that $\M/\P_\M \simeq M_2(k(\p))$. For these details see \cite[Paragraph 23.2.3]{voight}.
    
    \item When $B = M_2(F)$, we also consider $\J := \begin{psmallmatrix}
    \o & \o\\
    \p & \o
\end{psmallmatrix}$. The Jacobson radical of $\J$ is $\P_\J = \begin{psmallmatrix}
    \p & \o\\
    \p & \p
\end{psmallmatrix}$. If we let $\varpi_\J = \begin{psmallmatrix}
    0 & 1\\
    \varpi & 0
\end{psmallmatrix}$, then $\P_\J = \varpi_\J \J = \J \varpi_\J$. The normalizer of $\J$ in $\GL_2(F)$ is $K_\J = \varpi_\J^\Z \J^\times$. Note that $\J/\P_\J \simeq k(\p)\times k(\p)$. For these and other details, see \cite[Section 23.3]{voight}.

\item When $B = \D$, then $\O_\D$ is the unique maximal order of $B$. Recall that there is a valuation $v_\D: \D\rightarrow \Z \cup \{\infty\}$, with the usual properties:
\begin{itemize}
    \item $v_\D(0) = \infty$.
    \item $v_\D(xy) = v_\D(x) + v_\D(y)$.
    \item $v_\D(x + y) \geq \min(v_\D(x), v_\D(y))$.
\end{itemize} 
Explicitly, this valuation is given by the formula $v_\D(x) = v(\nrd(x))$. The maximal order can be described as $\O_\D = \{x \in \D: v_\D(x) \geq 0\}$. Its normalizer $K_\D$ is given by $K_\D = \D^\times$. The Jacobson radical $\P_\D$ of $\O_\D$ is $\P_\D = \{x \in \D: v_\D(x) \geq 1\}$. Note that $k(\P_\D) := \O_\D / \P_\D$ is a field extension of degree two of $k(\p)$. For these and other details, see \cite[Section 13.3]{voight}.
\end{itemize}

\begin{defi}\label{defi:principalneighbourhoods}
Let $\O \in \{\M, \J, \O_\D\}$. In each of the three cases above we let $U_\O := \O^\times$. The principal filtration of $U_\O$ is defined as follows:
\begin{itemize}
    \item $U_\O^0 := U_\O$.
    \item $U_\O^n := 1 + \P_\O^n$ for $n \geq 1$.
\end{itemize}
Similarly, we define a filtration of the normalizer $K_\O$ by letting:
\begin{itemize}
    \item $K_\O^{-1} := K_\O$.
    \item $K_\O^n := F^\times U_\O^n$ for $n \geq 0$. 
\end{itemize}
\end{defi}
Note that each subgroup $U_\O^n$ and $K_\O^n$ is normal in $K_\O$. We will also use the notation $U_\D^n = U^n_{\O_\D}$, as well as $K_\D^n = K_{\O_\D}^n$. Comparing these definitions with the groups $\O^\times(\p^n)$ introduced before we see that
\begin{equation}\label{eq:comparisonprincipalcongruencesubgroups}
    \M^1(\p^n) = U^n_\M, \quad \O_\D^1(\p^n) = U^{2n}_\D
\end{equation}
for any $n \geq 0$. On the other hand the groups $U^n_\M$ and $U^n_\J$ are related by the chain of inclusions
\begin{equation}\label{eq:chaininclusionshereditarymaximal}
    U_\M^{n + 1} \subsetneq U_\J^{2n + 1} \subsetneq U_\J^{2n} \subsetneq U_\M^{n}
\end{equation}
for any $n \geq 0$. We can use the filtration $K_\D^n$ to define a depth function on $\D^\times$ as follows.
\begin{defi}\label{defi:filtrationunitsdivision}
Given $\alpha \in \D^\times$, we define $d(\alpha) = \sup\{n \geq -1 : \alpha \in K_\D^{n}\}$.
\end{defi}
In particular, $d(\alpha) = -1$ holds precisely when $v_\D(\alpha)$ is odd, and $d(\alpha) = \infty$ holds precisely when $\alpha \in F^\times$. Also, the depth $d(\alpha)$ is even or odd according to whether the extension $F[\alpha]/F$ is unramified or ramified, respectively, see Section \ref{subsubsec:finishingtheproofoflocalmatching}. 

We are ready to state an explicit matching of certain test functions.

\begin{prop}\label{prop:mymatching}
Let $\alpha \in \D^\times \smallsetminus F$ and $n \geq 0$. Suppose that $d(\alpha)<n$. 
\begin{enumerate}
    \item [i)] If $d(\alpha)$ is even, then we have 
    \begin{equation*}
\varphi := \frac{1}{\vol(U_\D^{n})}\mathbbm{1}_{\alpha U_\D^{n}} \leftrightarrow f := \frac{1}{\vol(U_\M^{\lceil n/2 \rceil})}\mathbbm{1}_{\beta U_\M^{\lceil n/2 \rceil}},
\end{equation*}
where $\beta = \iota(\alpha)$ for any $\iota \in \Emb(F[\alpha]; M_2(F))$ satisfying $n_\M(\iota) = 0$.
    \item [ii)] If $d(\alpha)$ is odd, then we have 
    \begin{equation*}
\varphi := \frac{1}{\vol(U_\D^{n})}\mathbbm{1}_{\alpha U_\D^n} \leftrightarrow f := \frac{1}{\vol(U_\J^{n})}\mathbbm{1}_{\beta U_\J^{n}},
\end{equation*}
where $\beta = \iota(\alpha)$ for any $\iota \in \Emb(F[\alpha]; M_2(F))$ satisfying $n_\J(\iota) = 0$.
\end{enumerate}
\end{prop}

Note that an embedding $\iota$ as in the proposition exists in each of the three cases, see Lemma \ref{lem:optimalembeddingsKO}. The proof of Proposition \ref{prop:mymatching} occupies Section \ref{subsec:proofoflocalmatching}. The complementary case $d(\alpha) \geq n$ was already dealt with in \cite{tang2025}.
\begin{prop}[Tang, Wu, Yang, Yang]\label{prop:matchingtan}
Let $\alpha \in \D^\times$ and $n \geq 0$. Suppose that $d(\alpha) \geq  n$ and write $\alpha = t \alpha_0$, with $t \in F^\times$ and $\alpha_0 \in U_\D^n$.
\begin{enumerate}
    \item [i)] If $n = 0$, we have 
    \begin{equation*}
        \varphi := \frac{1}{\vol(U_\D)}\mathbbm{1}_{\alpha U_\D} \leftrightarrow f := \frac{2}{\vol(U_\M)}\mathbbm{1}_{t U_\M} - \frac{1}{\vol(U_\J)}\mathbbm{1}_{t U_\J}.
    \end{equation*}
    \item [ii)] If $n = 2k$ with $k \geq 1$, then we have 
    \begin{equation*}
        \varphi := \frac{1}{\vol(U_\D^{2k})}\mathbbm{1}_{\alpha U_\D^{2k}} \leftrightarrow f := \frac{2q}{(q - 1)\vol(U_\M^k)}\mathbbm{1}_{t U^k_\M} - \frac{q + 1}{(q - 1)\vol(U^{2k}_\J)}\mathbbm{1}_{t U^{2k}_\J}.
    \end{equation*}
    \item [iii)] If $n = 2k - 1$ with $k \geq 1$, then we have 
    \begin{equation*}
        \varphi := \frac{1}{\vol(U_\D^{2k-1})}\mathbbm{1}_{\alpha U_\D^{2k - 1}} \leftrightarrow f := -\frac{2}{(q - 1)\vol(U_\M^{k})}\mathbbm{1}_{t U^{k}_\M} + \frac{q + 1}{(q - 1)\vol(U^{2k - 1}_\J)}\mathbbm{1}_{t U^{2k-1}_\J}.
    \end{equation*}
\end{enumerate}
\end{prop}

\begin{proof}
This result follows inmediately by combining Proposition 2.4 and Remark 2.5 from \cite{tang2025} together with \eqref{eq:transformationorbital} above.
\end{proof}

Note that in the previous two propositions, if $\varphi$ is of level $\p^n$, then the matching function $f$ is of level $\p^{n + 1}$, recall Definition \ref{defi:levelfunction} as well as equations \eqref{eq:comparisonprincipalcongruencesubgroups} and \eqref{eq:chaininclusionshereditarymaximal}. We now explain how Proposition \ref{prop:mymatching} and Proposition \ref{prop:matchingtan} imply Proposition \ref{prop:explicitmatching}.

\begin{proof}[Proof of Proposition \ref{prop:explicitmatching} assuming Proposition \ref{prop:mymatching}.]
Suppose that $\varphi \in C_c^\infty(\D^\times)$ is of level $\p^n$. By definition it is invariant under translation by $U_\D^{2n}$, so we can find finitely-many elements $\alpha_i \in \D^\times$ and constants $c_i \in \CC$ such that  
\begin{equation*}
    \varphi = \sum_{i} c_i \mathbbm{1}_{\alpha_i U_\D^{2n}}.
\end{equation*}
By Proposition \ref{prop:mymatching} and Proposition \ref{prop:matchingtan} we can find functions $f_i \in C_c^\infty(\GL_2(F))$ of level $\p^{n + 1}$ such that $\mathbbm{1}_{\alpha_i U_\D^{2n}} \leftrightarrow f_i$. By linearity we deduce that 
\begin{equation*}
    \varphi \leftrightarrow f:= \sum_i c_i f_i.
\end{equation*}
Since orbital integrals are invariant under conjugation, after replacing $f$ by $x \mapsto \frac{1}{\vol(\GL_2(\o))}\int_{g \in \GL_2(\o)}f(g^{-1}xg) \, dg$ we can assume that $f$ is invariant under conjugation by $\GL_2(\o)$. This finishes the proof of the first claim in Proposition \ref{prop:explicitmatching}. To deduce the second claim, note that if $\varphi$ is supported in $\O_\D^\times$, then in the argument above we can assume that $\alpha_i \in \O_\D^\times$. By Propositions \ref{prop:mymatching} and \ref{prop:matchingtan} we can take $f_i \in C_c^\infty(\GL_2(\o))$, and then $f \in C_c^\infty(\GL_2(\o))$ as well.

For the claims about $\D^1$, suppose that we start with $\varphi \in C_c^\infty(\D^1)$ of level $\p^{n}$. Then we can find $\varphi_0 \in C_c^\infty(\D^\times)$ of level $\p^{n}$ such that $\varphi = \varphi_0 \mathbbm{1}_{\D^1}$. By the previous case there is $f_0 \in C_c^\infty(\GL_2(F))$ of level $\p^{n + 1}$, invariant under conjugation by $\GL_2(\o)$, and such that $\varphi_0 \leftrightarrow f_0$. Using \eqref{eq:matchingonBtimesimpliesB1} we deduce that $\varphi \leftrightarrow f:= f_0 \mathbbm{1}_{\SL_2(F)}$, which has the required properties. If $\varphi$ is supported in $\O_\D^1$, then we can take $\varphi_0$ supported in $\O_\D^\times$. Then we can take $f_0$ supported in $\GL_2(\o)$ and thus $f_0 \mathbbm{1}_{\SL_2(F)}$ is supported on $\SL_2(\o)$, as desired.
\end{proof}

\subsection{Explicit matching of local orbital integrals.}\label{subsec:proofoflocalmatching}

\subsubsection{Relation between orbital integrals and embedding numbers.}
We start by recalling the relation between orbital integrals and embedding numbers, expressed in Lemma \ref{lem:orbitalintegralOnadapted} below. 

We let $B$ be a quaternion algebra and $E$ a separable $F$-algebra of degree two. We know that the action of $B^\times$ on $\Emb(E; B)$ by conjugation is transitive \cite[Corollary 7.1.5]{voight}. Since the centralizer of the embedding $\iota \in \Emb(E; B)$ is $\iota(E^\times)$, we deduce that the map $g \mapsto g^{-1}\iota g$ induces a bijection
\begin{equation*}
    \iota(E^\times)\backslash B^\times \longleftrightarrow \Emb(E; B).
\end{equation*}
For an order $\O$ in $B$ and an order $S \subset E$ we let 
\begin{equation*}
    \Emb(S; \O) := \{\iota \in \Emb(E; B) : \iota^{-1}(\O) = S\}.
\end{equation*}
Equivalently, we have $\Emb(\o_{E, n}; \O) = \{\iota \in \Emb(E; B) : n_\O(\iota) = n\}$ for any $n\geq 0$. Clearly, we have a partition
\begin{equation*}
    \Emb(E; B) = \bigsqcup_{n \geq 0} \Emb(\o_{E, n}; \O),
\end{equation*}
where it could be the case that $\Emb(\o_{E, n}; \O)$ is empty for some values of $n$. Note that the action of $K_\O$ on $\Emb(E; B)$ preserves each of the sets $\Emb(\o_{E, n}; \O)$. Indeed, if $\iota_2 = g^{-1}\iota_1 g$ then $\iota_2^{-1}(\O) = \iota_1^{-1}(g\O g^{-1})$. If $g \in K_\O$, then by definition $g\O g^{-1} = \O$, and we deduce that $n_\O(\iota_1) = n_\O(\iota_2)$. This invariance property justifies the following definition.

\begin{defi}\label{defi:embeddingnumbers}
Let $S \subset E$ and $\O\subset B$ be orders and $\Gamma \subset K_\O$ a subgroup. Then we define $m(S; \O, \Gamma) := |\Emb(S; \O) / \Gamma|$.
\end{defi}
The quantities $m(S; \O, \Gamma)$ are usually called \emph{embedding numbers}, see \cite[Chapter 30]{voight}. Their computation is particularly simple when $\O \in \{\M, \J, \O_\D\}$ and $\Gamma = K_\O$.

\begin{lem}\label{lem:optimalembeddingsKO}
If $\O \in \{\M, \J, \O_\D\}$, then $m(\o_{E, n}; \O, K_\O) \leq 1$ for any $E$ and $n\geq 0$. We have $m(\o_{E, n}; \O, K_\O) = 0$ in exactly the following situations:
\begin{itemize}
    \item $B = \D$ and $E$ is split.
    \item $B = \D$, $E$ is a field and $n > 0$.
    \item $\O = \J$, $E$ is unramified and $n = 0$.
\end{itemize} 
\end{lem}
\begin{proof}
The lemma follows from the computations in \cite[Chapter 30]{voight}, see Proposition 30.5.3, Lemma 30.6.3 and Proposition 30.6.12.
\end{proof}

We define a map $h_B: B \rightarrow F\times F$ by the formula 
\begin{equation}\label{eq:defiofh}
    h_B(x) := (\trd_B(x), \nrd_B(x)),
\end{equation}
with the maps $h_E$ and $h_\O$ defined by the same formula but with $E$ and $\O$, respectively, as their domain. If $\iota \in \Emb(E; B)$ then we have $h_B(\iota(x)) = h_E(x)$. Thus, we will only keep the subscript when necessary (e.g. when computing fibers).

\begin{lem}\label{lem:conjugacyBtimesandKO}
Let $\O \in \{\M, \J, \O_\D\}$ and let $B$ be the quaternion algebra that contains $\O$. Then, we have the following:
\begin{enumerate}
    \item [i)] Two regular elements $x, y \in B_{\reg}$ are conjugate by $B^\times$ if and only if $h(x) = h(y)$. 
    \item [ii)] Two regular elements $x, y \in B_{\reg}$ are conjugate by $K_\O$ if and only $h(x) = h(y)$ and $n_\O(x) = n_\O(y)$.
\end{enumerate}
\end{lem}
\begin{proof}
If $x, y \in B_{\reg}$ then $h(x) = h(y)$ means exactly that $F[x] \simeq F[y]$ with an $F$-isomorphism sending $x$ to $y$. The first part of the proposition is \cite[Corollary 7.1.5]{voight}, and the second part follows from Lemma \ref{lem:optimalembeddingsKO} above.
\end{proof}

The following lemma expresses a connection between embedding numbers and orbital integrals of certain characteristic functions.
\begin{lem}\label{lem:orbitalintegralOnadapted}
Let $\O \in \{\M, \J, \O_\D\}$ and let $B$ the corresponding quaternion algebra. Suppose that $A \subset B_{\reg}^\times$ has the following properties:
\begin{enumerate}
    \item [i)] $A$ is invariant under conjugation by $K_\O$.
    \item [ii)] There is $n \geq 0$ such that for all $x\in A$ we have $n_\O(x) = n$.
\end{enumerate}
Let $h(A) \subset F\times F$ be the image of $A$ by $h$. Then 
\begin{equation*}
    \O_{B^\times}(\mathbbm{1}_A; E, x) = \frac{\vol_{B^\times}(\O^\times)}{\vol_{E^\times}(\o_{E, n}^\times)}\cdot m(\o_{E, n}; \O, \O^\times) \cdot \mathbbm{1}_{h(A)}(h(x)).
\end{equation*}
\end{lem}
\begin{proof}
By definition we have 
\begin{equation*}
    \O_{B^\times}(\mathbbm{1}_A; E, x) = \vol_{\iota(E^\times)\backslash B^\times}(\iota(E^\times)\backslash \{g \in B^\times: g^{-1}\iota(x)g \in A\}).
\end{equation*}
By the first part of Lemma \ref{lem:conjugacyBtimesandKO}, we know that $\{g \in B^\times: g^{-1}\iota(x)g \in A\} \neq \varnothing$ exactly when $h(x) \in h(A)$. Furthermore, if $g^{-1}\iota(x) g \in A$, then $n_\O(g^{-1}\iota g) = n$, by assumption. Thus, by the second part of Lemma \ref{lem:conjugacyBtimesandKO} and the assumption that $A$ is $K_\O$-invariant under conjugation it follows that 
\begin{equation*}
    \{g \in B^\times: g^{-1}\iota(x)g \in A\} = \{g \in B^\times: g^{-1}\iota g \in \Emb(\o_{E, n}; \O)\}.
\end{equation*}
Let us write $m = m(\o_{E, n}; \O, \O^\times)$ for the remainder of this proof. By definition of embedding numbers, we can find elements $g_i \in B^\times$, where $1 \leq i \leq m$, such that 
\begin{equation*}
    \{g \in B^\times: g^{-1}\iota g \in \Emb(\o_{E, n}; \O)\} = \bigsqcup_{i = 1}^{m} \iota(E^\times)g_i \O^\times.
\end{equation*}
Finally, by definition of the quotient measure we have 
\begin{equation}\label{eq:volumecertainquotient}
\begin{aligned}
\vol_{\iota(E^\times)\backslash B^\times}(\iota(E^\times)\backslash \iota(E^\times)g_i \O^\times) & = \vol_{\iota(E^\times)\backslash B^\times}(\iota(E^\times)\backslash \iota(E^\times)g_i \O^\times g_i^{-1})\\
& = \frac{\vol_{B^\times}(g_i \O^\times g_i^{-1})}{\vol_{\iota(E^\times)}(\iota(E^\times)\cap g_i \O^\times g_i^{-1})}\\
& = \frac{\vol_{B^\times}(\O^\times)}{\vol_{E^\times}(\o_{E, n}^\times)}.
\end{aligned}
\end{equation}
The last equality holds because $g_i^{-1}\iota g_i \in \Emb(\o_{E, n}; \O)$ implies that $\iota^{-1}(g_i \O^\times g_i^{-1}) = \o_{E, n}^\times$, by definition. The proof of the lemma is finished after summing \eqref{eq:volumecertainquotient} for $1 \leq i \leq m$
\end{proof}

In order to use Lemma \ref{lem:orbitalintegralOnadapted} in the proof of Proposition \ref{prop:mymatching} we still need to: compute certain embedding numbers; compute the images under $h$ of small principal neighbourhoods; understand how a small principal neighbourhood relates to a set satisfying the hypotheses of $A \subset B_{\reg}^\times$ in Lemma \ref{lem:orbitalintegralOnadapted}. We take care of these tasks in the next few paragraphs.

\subsubsection{Computation of embedding numbers.}
\begin{lem}\label{lem:optimalembeddingsOtimes}
Let $\O \in \{\M, \J, \O_\D\}$ and $\iota \in \Emb(\o_{E, n}; \O)$. Then 
\begin{equation*}
m(\o_{E, n}; \O, \O^\times) = [K_\O: (\iota(E^\times) \cap K_\O)\O^\times].    
\end{equation*}
\end{lem}
Note that $\iota(E^\times) \cap K_\O = \iota(\iota^{-1}(K_\O))$ and that $\iota^{-1}(K_\O)$ is independent of the choice of $\iota \in \Emb(\o_{E, n}; \O)$, due to the fact that all such embeddings are conjugate by $K_\O$.
\begin{proof}
This is exactly as in the proof of \cite[Lemma 30.3.14]{voight}. Since $m(\o_{E, n}; \O, K^\times) = 1$, the map $g \mapsto g^{-1}\iota g$ gives a bijection 
\begin{equation*}
    \iota(E^\times)\backslash \iota(E^\times)K_\O \longleftrightarrow \Emb(\o_{E, n}; \O),
\end{equation*}
which induces as well a bijection
\begin{equation*}
    \iota(E^\times)\backslash (\iota(E^\times)K_\O)/\O^\times \longleftrightarrow \Emb(\o_{E, n}; \O)/\O^\times,
\end{equation*}
given by the same map. 
Finally, note that that we have a natural bijection
\begin{equation*}
    \iota(E^\times)\backslash (\iota(E^\times)K_\O)/\O^\times \longleftrightarrow (\iota(E^\times) \cap K_\O)\backslash K_\O/\O^\times = (\iota(E^\times) \cap K_\O)\O^\times\backslash K_\O,
\end{equation*}
where the last equality holds since $\O^\times$ is normal in $K_\O$. This concludes the proof of the lemma.
\end{proof}

When $\iota^{-1}(K_\O) = E^\times$ the formula of  Lemma \ref{lem:optimalembeddingsOtimes} becomes $m(\o_{E, n}; \O, \O^\times) = [K_\O: \iota(E^\times)\O^\times]$. In Lemma \ref{lem:adaptedembeddingsorders} we will identify those cases in which $\iota^{-1}(K_\O) = E^\times$. For this purpose, it is convenient to look more explicitly at the structure of $\O \in \{\J, \O_\D\}$. 

Let $\O = \J$ if $B = M_2(F)$ and $\O = \O_\D$ if $B = \D$. Let $E_B \subset B$ be an $F$-algebra of degree two chosen as follows: if $B  \simeq M_2(F)$, then $E_B \simeq F \times F$, and if $B \simeq \D$ then $E_B/F$ is the unramified field extension of $F$ of degree two. 

\begin{lem}\label{lem:finerstructureramifiedorders}
When $\O \in \{\J, \O_\D\}$ there exists $\varpi_\O \in \O$ such that $\varpi_\O^2 = \varpi$ and such that $\varpi_\O^{-1}x \varpi_\O = \overline{x}$ holds for all $x \in E_B$. For such an element $\varpi_\O$ we have the decomposition
\begin{equation*}
    B = E_B \oplus \varpi_\O E_B.
\end{equation*}
For $x, y, z, w \in E_B$ we have the formula
\begin{equation*}
    (x + \varpi_\O y)(z + \varpi_\O w) = xz + \varpi \overline{y}w + \varpi_\O(yz + \overline{x}w),
\end{equation*}
as well as
\begin{equation}\label{eq:tracecyclicalgebra}
    \trd(x + \varpi_\O y) = \trd(x),\quad 
    \overline{x + \varpi_\O y} = \overline{x} - \varpi_\O y,
\end{equation}
and also
\begin{equation}\label{eq:normcyclicalgebra}
    \nrd(x + \varpi_\O y) = \nrd(x) -\varpi \nrd(y),
\end{equation}
Regarding the order $\O$ we have
\begin{equation}\label{eq:descriptionordercyclicpart1}
    \O = \o_{E_B} \oplus \varpi_\O \o_{E_B}, \quad \O^\times = \o_{E_B}^\times \oplus \varpi_\O \o_{E_B},
\end{equation}
as well as 
\begin{equation}\label{eq:descriptionordercyclicpart2}
    \P_\O = \p_{E_B} \oplus \varpi_\O \o_{E_B}, \quad K_\O = \varpi_\O^{\Z} \O^\times.
\end{equation}
\end{lem}
\begin{proof}
For the case $B = \D$ and $\O = \O_\D$ see \cite[Chapter 13]{voight}, in particular Theorem 13.3.11. For the case $B = M_2(F)$ and $\O = \J$ see \cite[Section 23.3]{voight}.  In this last case we can let $E_B := F \times F$ be embedded diagonally, define $\varpi_\O := \begin{psmallmatrix}
    0 & 1\\
    \varpi & 0
\end{psmallmatrix}$, and check the claims directly.  
\end{proof}

When $\O = \M$ we let $\varpi_\O := \varpi$. Note that in each of the three cases $\O \in \{\M, \J, \O_\D\}$ we have $K_\O = \varpi_\O^\Z \O^\times$.

For a field extension $E/F$ we let $e_E$ be the ramification degree. For $\O \in \{\M, \J, \O_\D\}$ we define  
\begin{equation*}
    e_\O := \begin{dcases}
        1, & \text{ if }\O = \M,\\
        2, & \text{ if }\O = \J \text{ or }\O = \O_\D.
    \end{dcases}
\end{equation*}

Recall that $x \in B$ is integral over $\o$ if $x$ satisfies a monic polynomial with coefficients in $\o$ \cite[Definition 10.3.1]{voight}. Equivalently, $x \in B$ is integral over $\o$ if and only if $(\trd(x), \nrd(x)) \in \o \times \o$ \cite[Corollary 10.3.6]{voight}.

\begin{lem}\label{lem:integralelementsnormalizer}
Let $\O \in \{\M, \J, \O_\D\}$. If $x \in K_\O$ is integral over $\o$, then $x \in \O$.  
\end{lem}
\begin{proof}
If $x \in K_{\O}$ we have $x = \varpi_\O^{k} u$ for some $k \in \Z$ and $u \in \O^\times$. Then $\nrd(x) \in \varpi^{2k/e_\O} \o^\times$. If $x$ is integral over $\o$ we must have $k \geq 0$. Since $\varpi_\O \in \O$, we conclude that $x \in \O$, as desired.
\end{proof}

\begin{lem}\label{lem:descriptionuniformizers}
Let $\O \in \{\J, \O_\D\}$ and $x \in K_\O$. Then, the following conditions are equivalent:
\begin{enumerate}
    \item [i)] $x \in \varpi_\O \O^\times$.
    \item [ii)] The characteristic polynomial of $x$ is an Eisenstein polynomial (over $\o$).
    \item [iii)] $F[x]/F$ is a ramified quadratic field extension and $x$ is a uniformizer of $F[x]$.
\end{enumerate}
\end{lem}
\begin{proof}
The equivalence of ii) and iii) is well-known, see \cite[Proposition 17]{serrelocalfields} and \cite[Proposition 18]{serrelocalfields}.

Let $E = E_B$ as in Lemma \ref{lem:finerstructureramifiedorders}. By direct computation we see that
\begin{equation}\label{auxeq:descriptionuniformizers}
\varpi_\O \O^\times = \varpi \o_E \oplus \varpi_\O \o_{E}^\times.
\end{equation} 
By formulas \eqref{eq:tracecyclicalgebra} and \eqref{eq:normcyclicalgebra} we deduce that if $x \in \varpi_\O \O^\times$, then the characteristic polynomial of $x$ is an Eisenstein polynomial. Thus, condition i) implies condition ii). 

Suppose now that $x$ satisfies condition ii). In particular $x$ is integral over $\o$, so we must have $x \in \O$ by Lemma \ref{lem:integralelementsnormalizer}. Let us write $x = a + \varpi_\O b$ with $a, b \in \o_E$. Condition ii) is equivalent to  
\begin{equation}\label{auxeq:conditionsforuniformizer}
    \trd(a) \in \p \quad \text{ and } \quad \nrd(a) - \varpi \nrd(b) \in \p \smallsetminus \p^2.
\end{equation}
The second condition implies that $\nrd(a) \in \p$. If we had $\nrd(a) \in \p \smallsetminus \p^2$, then the characteristic polynomial of $a$ would be an Eisenstein polynomial. Since, by hypothesis, either $E \simeq F\times F$ or $E/F$ is unramified, this is impossible. Thus, it must be that $\nrd(a) \in \p^2$ and $a\in \varpi \o_E$. Using the second condition in  \eqref{auxeq:conditionsforuniformizer} we deduce that $\nrd(b) \in \o^\times$ or, equivalently, that $b \in \o_E^\times$. Since $x = a + \varpi_\O b$, we deduce that $x \in \varpi \o_E \oplus \varpi_\O \o_{E}^\times$, and by equation \eqref{auxeq:descriptionuniformizers} we see that $x$ satisfies condition i), as desired.
\end{proof}

\begin{lem}\label{lem:adaptedembeddingsorders}
Suppose that $\O \in \{\O_\D, \M, \J\}$ and that $\Emb(\o_{E, n}; \O)$ is not empty. Let $\iota \in \Emb(\o_{E, n}; \O)$. Then we have $\iota^{-1}(K_\O) = E^\times$ precisely when $n = 0$, $E$ is a field and one of the following conditions hold:
\begin{enumerate}
    \item [i)] $\O \in \{\M, \O_\D\}$ and $E/F$ is unramified.
    \item [ii)] $\O \in \{\J, \O_\D\}$ and $E/F$ is ramified.
\end{enumerate}
\end{lem}

\begin{proof}
We first show that if $E^\times \subset \iota^{-1}(K_\O)$, then we must have $n = 0$. Indeed, note that $\o_E \cap E^\times$ is the set of elements of $E^\times$ which are integral over $\o$. Thus, if $E^\times \subset \iota^{-1}(K_\O)$, then by Lemma \ref{lem:integralelementsnormalizer} we have $\o_E \cap E^\times \subset \iota^{-1}(\O)$. One can check that in every case $\o_E \cap E^\times$ generates $\o_E$ as an $\o$-algebra, so we deduce that $\o_E \subset \iota^{-1}(\O)$ and thus $n=0$, as desired. Let us now treat each of the cases $\O \in \{\O_\D, \M, \J\}$.

The case $\O = \O_\D$ is very simple. Indeed, in this case we have $K_\O =  \D^\times$, so that $\iota(E^\times) \subset K_\O$ is clear. We also note that $\Emb(\o_E; \O_\D)$ is empty precisely when $E \simeq F \times F$, see Lemma \ref{lem:optimalembeddingsKO}. 

When $\O = \M$, then $K_\O = F^\times \M^\times$. If $E/F$ is the unramified quadratic extension, then $E^\times = F^\times \o_{E}^\times$, so we see that if $\iota \in \Emb(\o_E; \M)$ we have indeed $\iota(E^\times) \subset K_\O$. If $E$ is arbitrary and we have $\iota(E^\times)\subset K_\O$, then taking norms we see that $\nrd(E^\times)\subset \varpi^{2\Z}\o^\times = \nrd(K_\O)$. This implies that $E/F$ must be the unramified quadratic extension. 

Suppose now that $\O = \J$. Let $E/F$ be a ramified field extension, and let $\varpi_E$ be a uniformizer of $E$. By Lemma \ref{lem:optimalembeddingsKO} we know that $\Emb(\o_E; \J)$ is nonempty. Furthermore, by Lemma \ref{lem:descriptionuniformizers} we have $\iota(\varpi_E) \in \varpi_\O \O^\times \subset K_\O$. Since $E^\times = \varpi_E^{\Z}\o^\times$, we deduce that $\iota(E^\times)\subset K_\O$, as desired. If $E/F$ is an unramified field extension, then by Lemma \ref{lem:optimalembeddingsKO} we have $\Emb(\o_E; \J) = \varnothing$, so we can discard this case. Finally, suppose that $E = F \times F$ and let $\iota \in \Emb(\o_E; \J)$. Let $\alpha := (\varpi, 1) \in \o \times \o$. We claim that $\iota(\alpha) \notin K_\O$. Indeed, since $K_\O = \varpi_\O^{\Z}\O^\times$ and $\nrd(\alpha) = \varpi$, if we had $\iota(\alpha) \in K_\O$ we would deduce that $\iota(\alpha) \in \varpi_\O \O^\times$. But by Lemma \ref{lem:descriptionuniformizers} this would imply that $E$ is a ramified field extension, contradicting the assumption that $E  = F\times F$. 
\end{proof}

\subsubsection{Characteristic polynomials in small neighbourhoods.}
Let us write
\begin{equation*}
\begin{aligned}
Z_u & := \{(t, n) \in \o\times \o: X^2 - tX + n \modulo \p \text{ is irreducible in }k(\p)[X] \},\\
Z_r & := \{(t, n) \in \o \times \o: X^2 - tX + n \in \o[X] \text{ is an Eisenstein polynomial}\} = \p \times (\p \smallsetminus \p^2).
\end{aligned}
\end{equation*}
Let $x \in B$ and recall the function $h_B$ from \eqref{eq:defiofh}. We know that \cite[Ch. I, §6]{serrelocalfields}: 
\begin{enumerate}
    \item [i)] If $h(x) \in Z_u$, then $F[x]\subset B$ is a quadratic unramified field extension of $F$ and $\o[x]$ is its maximal order. 
    \item [ii)] If $h(x) \in Z_r$, then $F[x]\subset B$ is a ramified field extension, $\o[x]$ is its maximal order and $x$ a uniformizer of $F[x]$.
\end{enumerate}
\begin{lem}\label{lem:certainelementsareoptimal}
If $x\in \O$ and $h(x) \in Z_u \cup Z_r$, then $n_\O(x) = 0$.
\end{lem}
\begin{proof}
We have $\o[x] \subset \O$ and the assumption guarantees that $\o[x]$ is the maximal order of $F[x]$, thus we must have $\O \cap F[x] = \o[x]$, as desired. 
\end{proof}

\begin{lem}\label{lem:imageofhhereditaryorders}
We have $Z_u \cup Z_r \subset h(\M)\cap h(\O_\D)$, as well as $Z_r \subset h(\J)$.
\end{lem}
\begin{proof}
For $(t, n) \in F \times F$ define $E := F[X]/(X^2 - tX + n)$ and let $\alpha := X \modulo (X^2 - tX + n)$. Note that $E$ is a separable $F$-algebra if and only if $t^2 - 4n \neq 0$. Also, note that $(t, n) \in h(\O)$ is equivalent to the existence of an embedding $\iota \in \Emb(E; B)$ such that $\iota(\o[\alpha])\subset \O$. If $(t, n) \in Z_u \cup Z_r$, we deduce, by Lemma \ref{lem:certainelementsareoptimal}, that $(t, n) \in h(\O)$ if and only if $m(\o_{E}, \O; K_\O) > 0$. The proof of the lemma is finished by applying Lemma \ref{lem:optimalembeddingsKO}.
\end{proof}

\begin{lem}\label{lem:preimagesramifiedunramified}
If $\O \in \{\J, \O_\D\}$ we have $h_\O^{-1}(Z_r) = \varpi_\O \O^\times$. If $\O = \O_\D$, we have $h_{\O_\D}^{-1}(Z_u) = U_\D - \o^\times U_\D^1$. 
\end{lem}
\begin{proof}
The first claim is a reformulation of Lemma \ref{lem:descriptionuniformizers}. The second follows from \eqref{eq:descriptionordercyclicpart1} and \eqref{eq:descriptionordercyclicpart2}. 
\end{proof}

We state a consequence of Hensel's lemma which is convenient for the computations below.
\begin{lem}\label{lem:auxiliaryhensel}
Let $f: \o^n \rightarrow \o^m$ be a map of the form $f(x) = Ax + \varpi P(x)$, where $A \in M_{m \times n}(\o)$ and $P: \o^n \rightarrow \o^m$ is a polynomial map with coefficients in $\o$. Then $f$ is surjective if and only if $x\mapsto Ax$ is surjective modulo $\p$.
\end{lem}
\begin{proof}
The fact that if $f$ is surjective then $x \mapsto Ax$ is surjective modulo $\p$ is clear from the shape of $f$. For the other direction, using that $\o$ is local one can perform row and column operations to bring $A$ to the shape $A = (I_m \, B)$, where $I_m$ is the $m\times m$ identity matrix and $B \in M_{m \times (n - m)}(\o)$. With this assumption the result follows by applying Hensel's lemma one variable at a time.
\end{proof}

Recall Definition \ref{defi:principalneighbourhoods}.

\begin{lem}\label{lem:charpolyunramified}
Let $\O \in \{\M, \O_\D\}$ and $\alpha \in h_\O^{-1}(Z_u)$. Then, for $n \geq 1$ we have 
\begin{equation*}
    h(\alpha U_\O^n) = h(\alpha) + \p^{\lceil n/e_\O\rceil} \times \p^{\lceil n/e_\O\rceil}.
\end{equation*}
\end{lem}

\begin{proof}
Suppose first that $\O = \M$. We have $\alpha U_\M^n = \alpha + \P_\M^n$. Writing $\beta = \alpha + \varpi^n x$, then 
\begin{equation}\label{auxeq:proofofunramifiedcharacteristic}
    \trd(\beta) = \trd(\alpha) + \varpi^n \trd(x), \quad \nrd(\beta) = \nrd(\alpha) + \varpi^n \trd(\overline{\alpha}x) + \varpi^{2n}\nrd(x).
\end{equation}
Thus, $h(\alpha U_\M^n) \subset h(\alpha) + \p^n \times \p^n$. By Lemma \ref{lem:auxiliaryhensel}, to prove the equality it is enough to show that $g: \O_\M \rightarrow \o^2$ is surjective, where 
\begin{equation*}
g(x) = (\trd(x), \trd(\overline{\alpha}x)).
\end{equation*}
In fact, since $F[\alpha]/F$ is unramified and $\o[\alpha]$ is its ring of integers, we know by the determinantal definition of the discriminant that $g(\o[\alpha]) = \o^2$, see \cite[Chapter III]{serrelocalfields}, in particular Proposition 3 in Section 2.

Suppose now that $\O = \O_\D$. As before we have $\alpha U_\D^n = \alpha + \P_\D^n$. If $\beta = \alpha + x$, with $x \in \P_\D^n$, then
\begin{equation*}
    \trd(\beta) = \trd(\alpha) +\trd(x), \quad \nrd(\beta) = \nrd(\alpha) + \trd(\overline{\alpha}x) + \nrd(x).
\end{equation*} 
We have $\trd(x), \trd(\overline{\alpha}x) \in \p^{\lceil n/2\rceil}$, while $\nrd(x) \in \p^n$. Thus,
\begin{equation*}
    h(\alpha U_\D^n) \subset h(\alpha) + \p^{\lceil n/2\rceil} \times \p^{\lceil n/2 \rceil}.
\end{equation*}
It remains to prove the opposite containment, for which it is enough to assume that $n = 2m$ is even. If we write $\beta = \alpha + \varpi^m x$ with $x \in \O_\D$, then equation \eqref{auxeq:proofofunramifiedcharacteristic} holds and the proof is exactly as before. 
\end{proof}

Now we move on to elements $\alpha \in h_\O^{-1}(Z_r)$. 

\begin{lem}\label{lem:charpolyramified}
Let $\O \in \{\J, \O_\D\}$ and $\alpha \in h_\O^{-1}(Z_r)$. Then, we have: 
\begin{enumerate}
    \item [i)]$h(\alpha U_\O) = Z_r = \p \times (\p \smallsetminus \p^2)$,
    \item [ii)]$h(\alpha U_\O^{2n - 1}) = h(\alpha) + \p^n \times \p^{n + 1}$ for $n \geq 1$, and
    \item [iii)]$h(\alpha U_\O^{2n}) = h(\alpha) + \p^{n + 1}\times \p^{n + 1}$ for $n \geq 1$.
\end{enumerate}
\end{lem}
\begin{proof}
Recall that $U_\O = \O^\times$. By Lemma \ref{lem:preimagesramifiedunramified} we know that $h_\O^{-1}(Z_r) = \varpi_\O U_\O$. Since $\alpha \in h_\O^{-1}(Z_r)$ by hypothesis, we deduce that $h_\O^{-1}(Z_r) = \alpha U_\O$. On the other hand, by Lemma \ref{lem:imageofhhereditaryorders} we know that $Z_r \subset h(\O)$. Part i) follows from these observations. 

We now deal with part ii). Note that $\alpha U_\O^{2n-1} = \alpha + \P_\O^{2n} = \alpha + \varpi^n \O$. Write $\beta \in \alpha U_\O^{2n-1}$ as $\beta = \alpha + \varpi^n y$, with $y \in \O$. Furthermore, recalling \eqref{eq:descriptionordercyclicpart1} and \eqref{auxeq:descriptionuniformizers} we write
\begin{equation}\label{auxeq:decompositionelements}
    \alpha = \varpi \gamma + \varpi_\O \delta, \quad y = z + \varpi_\O w
\end{equation}
with $\gamma \in \o_E, \delta \in \o_E^\times$ and $z, w \in \o_E$. Recall that $E = F\times F$ if $\O = \J$, and $E/F$ is the unramified quadratic extension if $\O = \O_\D$. Using \eqref{eq:tracecyclicalgebra} and \eqref{eq:normcyclicalgebra} we see that 
\begin{equation*}
    \trd(y) = \trd(z), \quad \trd(\overline{\alpha}y) = \varpi \trd(\overline{\gamma}z) - \varpi \trd(\overline{\delta}w),\quad \nrd(y) = \nrd(z)-\varpi \nrd(w).
\end{equation*}
Therefore, 
\begin{equation*}
\begin{aligned}
\trd(\beta) & = \trd(\alpha) + \varpi^n \trd(z),\\
\nrd(\beta) & = \nrd(\alpha) + \varpi^{n + 1} (\trd(\overline{\gamma}z)- \trd(\overline{\delta}w)) + \varpi^{2n}\nrd(z) - \varpi^{2n + 1}\nrd(w).
\end{aligned}
\end{equation*}
From this equation we see that $h(\alpha U_\O^{2n-1})\subset h(\alpha) + \p^n \times \p^{n + 1}$. To show equality in this inclusion, we need to prove that the map $g: \o_E^2 \rightarrow \o^2$ given by
\begin{equation*}
    g(z, w) = (\trd(z), \trd(\overline{\gamma}z - \overline{\delta}w) + \varpi^{n-1}\nrd(z) - \varpi^{n}\nrd(w))
\end{equation*}
is surjective. Let $(t, n) \in \o^2$. Since $\trd(\o_E) = \o$, we can always choose $z$ so that $\trd(z) = t$. After $z$ is fixed, the task reduces to showing that the map $g': \o_E \rightarrow \o$ given by
\begin{equation*}
    g'(w) = \trd(\overline{\delta}w) + \varpi^n \nrd(w)
\end{equation*}
is surjective. Since $\delta \in \o_E^\times$, we have $\trd(\overline{\delta}\o_E) = \o$, so the surjectivity of $g'$ follows from Lemma \ref{lem:auxiliaryhensel}. This finishes the proof of part ii).

For part iii), we have $\alpha U_\O^{2n} = \alpha + \P_\O^{2n + 1}$. If we write $\beta = \alpha + \varpi^n \varpi_\O y$, with $y \in \O$, then 
\begin{equation*}
    \trd(\beta) = \trd(\alpha) + \varpi^n \trd(\varpi_\O y), \quad \nrd(\beta) = \nrd(\alpha) + \varpi^n \trd(\overline{\alpha}\varpi_\O y) - \varpi^{2n + 1} \nrd(y).
\end{equation*}
Since $\trd(\P_\O) = \p$, it is clear that $h(\alpha U_\O^{2n}) \subset h(\alpha) + \p^{n + 1} \times \p^{n + 1}$. To check equality, we use the decomposition \eqref{auxeq:decompositionelements} and find that
\begin{equation*}
\begin{aligned}
    \trd(\beta) & = \trd(\alpha) + \varpi^{n + 1} \trd(w),\\
    \nrd(\beta) & = \nrd(\alpha) -\varpi^{n+1}\trd(\overline{\delta}z) + \varpi^{n + 2}\trd(\overline{\gamma}w) - \varpi^{2n + 1}(\nrd(z) - \varpi \nrd(w)).
\end{aligned}
\end{equation*} 
By Lemma \ref{lem:auxiliaryhensel}, it is now enough to show that the map $g: \o_E^2 \rightarrow \o^2$ given by $g(z, w) = (\trd(w), -\trd(\overline{\delta}z))$ is surjective. This is clear since $\delta \in \o_E^\times$ and $\trd(\o_E) = \o$.
\end{proof}

\subsubsection{Centralizers.}

In this section we compute the normalizer in $K_\O$ of the principal neighbourhood $\alpha U_\O^n$ as in Lemma \ref{lem:charpolyunramified} or Lemma \ref{lem:charpolyramified}. We start with a very basic group theoretic fact. 
\begin{lem}\label{lem:basicgrouptheoreticlemma}
Let $H_1, H_2$ and $N$ be subgroups of $G$. Assume that $H_1 \subset H_2$ and that $N$ is normal in $G$. Suppose that $H_1 \cap N = H_2 \cap N$ and $H_1N = H_2 N$. Then $H_1 = H_2$.
\end{lem}
\begin{proof}
Let $h_2 \in H_2$. Then there is $h_1 \in H_1$ and $n \in N$ such that $h_2 = h_1 n$. Thus $h_1^{-1}h_2 = n \in N\cap H_2 = N \cap H_1$. So $h_1^{-1}h_2 \in H_1$ and, thus, $h_2 \in H_1$ as well. 
\end{proof}

Recall Lemma \ref{lem:finerstructureramifiedorders} and let $E = E_\D \subset \D$ be the unramified quadratic extension of $F$. We have a natural commutative diagram,
\begin{equation}\label{eq:diagramunramifiedinD}
\begin{tikzcd}
\o_E \arrow[rr, hook] \arrow[d]   &  & \O_\D \arrow[d] \\
\o_{E}/\p_{E} \arrow[rr, "\simeq"] &  & \O_\D/\P_\D,
\end{tikzcd}
\end{equation} 
where the bottom horizontal arrow is an isomorphism. Note that $\o_{E}/\p_{E} \simeq \O_\D /\P_\D \simeq \mathbb{F}_{q^2}$ is the degree two extension of $k(\p) \simeq \mathbb{F}_q$. We let $\sigma$ denote the Frobenius automorphism of $\mathbb{F}_{q^2}$ over $\mathbb{F}_q$. It has order two, and $\sigma x = x$ if and only $x \in \mathbb{F}_q$. Note that if $x \in \O_\D$, then
\begin{equation}\label{eq:standardinvolutionandfrobenius}
    \overline{x} \modulo \P_\D =  \sigma (x \modulo \P_\D)
\end{equation}
where $x \mapsto \overline{x}$ is the standard involution. It follows that if $x \in \O_\D$, then $x \equiv \overline{x} \modulo \P_\D$ holds if and only if $x \in \o + \P_\D$. Recall that $\D^\times$ normalizes $\O_\D$. 

\begin{lem}\label{lem:frobeniusconjugationquaternion}
Let $\alpha \in \D^\times$ and $x \in \O_\D$. Then 
\begin{equation}\label{auxeq:frobenius}
    \alpha x \alpha^{-1} \modulo \P_\D = \sigma^{v_\D(\alpha)} (x \modulo \P_\D).
\end{equation}
\end{lem}
\begin{proof}
Since $\O_\D/\P_\D$ is commutative, the statement is clear when $\alpha \in \O_\D^\times$. Furthermore, since $F^\times$ centralizes $\D^\times$ we see that the statement is true when $v_\D(\alpha) \in 2\Z$. We know that $v_\D(\varpi_{\O_\D}) = 1$ and $\varpi_{\O_\D} x \varpi_{\O_\D}^{-1} = \overline{x}$ for $x \in \o_{E}$. The proof is finished by comparing with \eqref{eq:standardinvolutionandfrobenius}.
\end{proof}

If $\varphi$ is any function on $B^\times$, we define $S_{K_\O}(\varphi) := \{x \in K_\O^\times: \varphi^x = \varphi\}$. Since $K_\O$ normalizes $U_\O^n$, we know that if $g_1, g_2 \in K_\O$, then
\begin{equation}\label{eq:rigiditynormalizerprincipalneighbourhood}
    g_1^{-1} \alpha U_\O^n g_1 \cap g_2^{-1}\alpha U_\O^n  g_2 \neq \varnothing \implies g_1^{-1}\alpha U_\O^n g_1 = g_2^{-1}\alpha U_\O^n g_2.
\end{equation}
Therefore, if $\alpha \in B^\times$ and $n \geq 0$ we have $S_{K_\O}(\mathbbm{1}_{\alpha U_\O^n}) = \{x \in K_\O^\times : x^{-1}\alpha x \in \alpha U_\O^n\}$. 
\begin{lem}\label{lem:centralizer}
Let $\alpha U_\O^n$ be as in Lemma \ref{lem:charpolyunramified} or Lemma \ref{lem:charpolyramified}. Then $S_{K_\O}(\mathbbm{1}_{\alpha U_\O^n}) = F[\alpha]^\times U_\O^n$. 
\end{lem}
\begin{proof}
When $\O \in \{\M, \J\}$ this is proved in \cite[Lemma 16.2]{bushnellhenniart}. We will adapt their proof to the case $B = \D$ and $\O = \O_\D$. For the remainder of the proof we let $\varphi := \mathbbm{1}_{\alpha U_\O^n}$. Define $A_n := F[\alpha]^\times U_\D^n$. It is clear that $A_n \subset S_{\D^\times}(\varphi)$, and we want to show that this inclusion is an equality.  

Suppose first that $\alpha \in h_\O^{-1}(Z_r)$. Since $v_\D(\alpha) = 1$ we have $A_n U_\D = \D^\times= S_{\D^\times}(\varphi) U_\D$. Thus, by Lemma \ref{lem:basicgrouptheoreticlemma} it is enough to show that $A_n \cap U_\D = S_{\D^\times}(\varphi) \cap U_\D$. We know that $x \in S_{\D^\times}(\varphi)$ if and only if $x^{-1}\alpha x \in \alpha U_\D^n$, which if $x \in U_\D$ is also equivalent to 
\begin{equation}\label{auxeq:centralizer}
    \alpha x \alpha^{-1} \equiv x \modulo \P_\D^n.
\end{equation}
It is enough to show, by induction on $n\geq 1$, that if $x \in \O_\D$ satisfies \eqref{auxeq:centralizer}, then $x \in \o[\alpha] + \P_\D^n$. We first establish the base case $n = 1$. Since $v_\D(\alpha) = 1$, by Lemma \ref{lem:frobeniusconjugationquaternion} we have 
\begin{equation*}
    \alpha x \alpha^{-1} \equiv \overline{x} \modulo \P_\D. 
\end{equation*}  
If \eqref{auxeq:centralizer} holds, then we deduce that 
\begin{equation*}
    x \equiv \overline{x} \modulo \P_\D.
\end{equation*}
This implies that $x \in \o + \P_\D = \o[\alpha] + \P_\D$, proving the case $n = 1$.

Suppose that $n \geq 2$ and that the claim has been proven for $1 \leq m < n$. Suppose that $x \in \O_\D$ satisfies \eqref{auxeq:centralizer}. Then, by the induction hypothesis we have $x \in \o[\alpha] + \P_\D^{n-1}$. Writing $x = \beta + \alpha^{n-1} y$ with $\beta \in \o[\alpha]$ and $y \in \O_\D$, equation \eqref{auxeq:centralizer} reduces to 
\begin{equation*}
    \alpha y \alpha^{-1}  \equiv y \modulo \P_\D.
\end{equation*}
By the case $n = 1$ we have $y \in \o[\alpha] + \P_\D$, so 
\begin{equation*}
    x \in \o[\alpha] + \alpha^{n-1}(\o[\alpha] + \P_\D) \subset \o[\alpha] + \P_\D^n,
\end{equation*}
as desired.

Suppose now that $\alpha\in h_\O^{-1}(Z_u)$. In this case we have $A_n U_\D = \{x \in \D^\times: v_\D(x) \in 2\Z\}$. Since $\o[\alpha] \rightarrow \O_\D/\P_\D$ is surjective, Lemma \ref{lem:frobeniusconjugationquaternion} guarantees that if $x \in S_{\D^\times}(\varphi)$, then $v_\D(x)$ is even. Therefore, we have $A_n U_\D = S_{\D^\times}(\varphi) U_\D$. By Lemma \ref{lem:basicgrouptheoreticlemma} we are left to show that $A_n \cap U_\D = S_{\D^\times}(\varphi) \cap U_\D$.

In the same way as above, the problem reduces to showing the following claim: if equation \eqref{auxeq:centralizer} is true for some $x \in \O_\D$ and some $n \geq 1$, then $x \in \o[\alpha] + \P_\D^n$. We prove this claim by induction on $n$. The case $n = 1$ is clear since $\o[\alpha] + \P_\D = \O_\D$. For the case $n = 2$ write $x = \beta + y$ with $\beta \in \o[\alpha]$ and $y \in \P_\D$. Equation \eqref{auxeq:centralizer} reduces to $\alpha y \alpha^{-1} \equiv  y \modulo \P_\D^2$, which is equivalent to $y^{-1}\alpha y \equiv \alpha \modulo \P_\D$. Lemma \ref{lem:frobeniusconjugationquaternion} implies that $v_\D(y)$ is even. Since we already knew that $y \in \P_\D$, we deduce that $y \in \P_\D^2$, as desired. Suppose now that \eqref{auxeq:centralizer} holds for some $n \geq 3$, and that we have proved the claim for $1 \leq m < n$. Then, using the case $m = 2\lfloor (n-1)/2\rfloor < n$ we deduce that $x = \beta + \varpi^{\lfloor \frac{n-1}{2}\rfloor} y$ with $\beta \in \o[\alpha]$ and $y \in \O_\D$. Equation \eqref{auxeq:centralizer} reduces to $\alpha y \alpha^{-1}\equiv y \modulo \P_\D^{n - 2\lfloor \frac{n-1}{2}\rfloor}$. Since $1 \leq n - 2\lfloor \frac{n-1}{2}\rfloor \leq 2$, using the base cases of the induction we deduce that $y \in \o[\alpha] + \P_\D^{n - 2\lfloor \frac{n-1}{2}\rfloor}$, so that 
\begin{equation*}
    \alpha \in \o[\alpha] + \varpi^{\lfloor \frac{n-1}{2}\rfloor} \left(\o[\alpha] + \P_\D^{n - 2\lfloor \frac{n-1}{2}\rfloor}\right) \subset \o[\alpha] + \P_\D^n,
\end{equation*}
as desired.
\end{proof}

\begin{lem}\label{lem:indexofcentralizer}
Let $\alpha U_\O^n$ be as in Lemma \ref{lem:charpolyunramified} or Lemma \ref{lem:charpolyramified}. Then
\begin{equation*}
    [K_\O: F[\alpha]^\times U_\O^n] = [K_\O: F[\alpha]^\times \O^\times] \cdot \frac{[U_\O: U_\O^n]}{[U_{F[\alpha]}: U_{{F[\alpha]}}^{\lceil n \times (e_{F[\alpha]}/e_\O)\rceil}]}.
\end{equation*}
\end{lem}
\begin{proof}
Let $A_n = F[\alpha]^\times U_\O^n$. We have $[K_\O: A_n] = [K_\O: A_n\O^\times][\O^\times : A_n \cap \O^\times]$. 
Since, $A_n \cap \O^\times = \o[\alpha]^\times U_\O^n$ we see that 
\begin{equation*}
    [\O^\times : A_n \cap \O^\times] = \frac{[\O^\times: U_\O^n]}{[ \o[\alpha]^\times U_\O^n: U_\O^n]}.
\end{equation*}
To compute the denominator we note that $\o[\alpha]^\times U_\O^n/U_\O^n \simeq \o[\alpha]^\times / \left(\o[\alpha]^\times \cap U_\O^n\right)$ and that we have $\o[\alpha]^\times \cap U_\O^n = U_{{F[\alpha]}}^{\lceil n \times (e_{F[\alpha]}/e_\O)\rceil}$, which can be checked directly.
\end{proof}

\subsubsection{Orbital integrals of characteristic functions of principal neighbourhoods.}
We are now ready to compute orbital integrals for principal neighbourhoods. 

\begin{prop}\label{prop:orbitalintegralneighbourhood}
Let $\alpha U_\O^n$ as in Lemma \ref{lem:charpolyunramified} or Lemma \ref{lem:charpolyramified}. Then 
\begin{equation}\label{eq:orbitalintegralneighbourhood}
    \O_{B^\times}(\mathbbm{1}_{\alpha U_\O^n}; E, x) = \frac{\vol_{B^\times}(U_\O^n)}{\vol_{E^\times}(U_E^{\lceil n \times (e_{E}/e_\O)\rceil})} \cdot \mathbbm{1}_{h(\alpha U_\O^n)}(h(x)).
\end{equation} 
\end{prop}

\begin{proof}
Let $A := h_\O^{-1}(h(\alpha U_\O^n))$, so that in particular we have $h(A) = h(\alpha U_\O^n)$. By Lemma \ref{lem:certainelementsareoptimal}, if $x \in A$ then $n_\O(x) = 0$. It is also clear that $A$ is $K_\O$-invariant under conjugation. Thus, the set $A$ satisfies the hypothesis of Lemma \ref{lem:orbitalintegralOnadapted}, so that 
\begin{equation}\label{eq:orbitalintegralofA}
\O(\mathbbm{1}_A; E, x) = \frac{\vol_{B^\times}(\O^\times)}{\vol_{E^\times}(\o_E^\times)}\cdot m(\o_{E}; \O, \O^\times) \cdot \mathbbm{1}_{h(A)}(h(x)).
\end{equation} 
By Lemma \ref{lem:conjugacyBtimesandKO}, $A$ is the set of elements of $B_{\reg}$ which are $K_\O$-conjugate to some element in $\alpha U_\O^n$. We can assume that $h(x) \in h(\alpha U_\O^n)$, since otherwise both sides of \eqref{eq:orbitalintegralneighbourhood} vanish. In this case there is $\iota \in \Emb(\o_E; \O)$ such that $\iota(x) \in \alpha U_\O^n$, so without loss of generality we can assume that $E = F[\alpha]$ and $x = \alpha$. By equation \eqref{eq:rigiditynormalizerprincipalneighbourhood} and Lemma \ref{lem:centralizer} we find that
\begin{equation*}
    \mathbbm{1}_{A}(x) = \sum_{g \in F[\alpha]^\times U_\O^n \backslash K_\O} \mathbbm{1}_{\alpha U_\O^n}(g^{-1}xg).
\end{equation*}
Since orbital integrals are conjugation-invariant, we deduce from \eqref{eq:orbitalintegralofA} that 
\begin{equation}\label{auxeq:orbitalintegralneighbourhood}
\O(\mathbbm{1}_{\alpha U_\O^n}; E, x) =  \frac{\vol_{B^\times}(\O^\times)}{[K_\O: F[\alpha]^\times U_\O^n]\vol_{E^\times}(\o_E^\times)}\cdot m(\o_{E}; \O, \O^\times) \cdot \mathbbm{1}_{h(A)}(h(x)).
\end{equation}
By Lemma \ref{lem:optimalembeddingsOtimes} and Lemma \ref{lem:adaptedembeddingsorders} we know that
\begin{equation*}
    m(\o_E; \O, \O^\times) = [K_\O: \iota(E^\times)\O^\times]. 
\end{equation*}
The proof of the lemma is finished by appealing to Lemma \ref{lem:indexofcentralizer} and simplifying the resulting formula. 
\end{proof}

\begin{lem}\label{lem:volumeprincipalsubgroups}
If $n \geq 1$ and $\O \in \{\M, \J, \O_\D\}$ then
\begin{equation*}
    \vol(U_\O^n) = q^{\frac{4}{e_\O}}\vol(U_\O^{n + 1}).
\end{equation*}    
\end{lem}
\begin{proof}
Since $n \geq 1$, the map $x \mapsto 1 + x$ induces an isomorphism of groups $\P_\O^{n}/\P_\O^{n + 1} \simeq U^{n}_\O/ U^{n + 1}_\O$. Also, the map $x \mapsto \varpi_\O^n x$ induces an isomorphism $\O/\P_\O \simeq \P_\O^{n}/\P_{O}^{n+1}$ of $k(\p)$-vector spaces. Finally, since $\dim_{k(\p)}(\O/\P_\O) = 4/e_\O$ we conclude that $[U^{n}_\O: U^{n + 1}_\O] = q^{4/e_\O}$, which is equivalent to the assertion of the lemma.
\end{proof}

We are now ready to give the proof of Proposition \ref{prop:mymatching}.

\subsubsection{Proof of Proposition \ref{prop:mymatching}.}\label{subsubsec:finishingtheproofoflocalmatching}
Recall Definition \ref{defi:filtrationunitsdivision}. Let $\alpha \in \D \smallsetminus F$ and $n \geq 0$ such that $d(\alpha)< n$. Suppose first that $d(\alpha) = -1$. Then, there exists $t \in F^\times$ such that $\alpha = t \alpha_0$, where $v_\D(\alpha_0) = 1$. By Lemma \ref{lem:preimagesramifiedunramified} we have $\alpha_0 \in h_{\O_\D}^{-1}(Z_r)$. Let $\iota \in \Emb(\o[\alpha_0]; \J)$. We know that $\iota$ exists by Lemma \ref{lem:optimalembeddingsKO}. By Lemma \ref{lem:charpolyramified} and Proposition \ref{prop:orbitalintegralneighbourhood} we deduce the matching
\begin{equation*}
\frac{1}{\vol(U_\D^n)}\mathbbm{1}_{\alpha_0 U_\D^n} \leftrightarrow \frac{1}{\vol(U_\J^n)}\mathbbm{1}_{\iota(\alpha_0) U_\J^n},
\end{equation*}
and by equation \eqref{eq:transformationorbital} we see that 
\begin{equation*}
    \frac{1}{\vol(U_\D^n)}\mathbbm{1}_{\alpha U_\D^n} \leftrightarrow \frac{1}{\vol(U_\J^n)}\mathbbm{1}_{\iota(\alpha) U_\J^n},
\end{equation*}
as desired. Suppose now that $d(\alpha) = 2k + 1$ with $k \geq 0$. Then $\alpha = t(1 + \varpi^k \alpha_0)$ for some $t \in F^\times$ and $\alpha_0 \in h_\D^{-1}(Z_r)$. Let $\iota \in \Emb(\o[\alpha_0]; \J)$.
By Lemma \ref{lem:preimagesramifiedunramified} we have $\alpha_0 \in \varpi_\D \O_\D^\times$ and $\iota(\alpha_0) \in \varpi_\J \J^\times$. Note that $\P_\O^\ell = \varpi_\O^{\ell}\O$ for any $\ell$ . 
It follows that 
\begin{equation}\label{eq:transformationsetramified}
    \alpha U_\D^n = t(1 + \varpi^k \alpha_0 U_\D^{n-2k-1}), \quad \iota(\alpha)U_\J^n = t(1 + \varpi^k \iota(\alpha_0)U_\J^{n-2k-1}).
\end{equation}
By Lemma \ref{lem:charpolyramified} and Proposition \ref{prop:orbitalintegralneighbourhood} we have the matching
\begin{equation*}
\frac{1}{\vol(U_\D^{n-2k-1})}\mathbbm{1}_{\alpha_0 U_\D^{n-2k-1}} \leftrightarrow \frac{1}{\vol(U_\J^{n-2k-1})}\mathbbm{1}_{\iota(\alpha_0) U_\J^{n-2k-1}}.
\end{equation*}
By equations \eqref{eq:transformationorbital} and \eqref{eq:transformationsetramified} and Lemma \ref{lem:volumeprincipalsubgroups} we see that 
\begin{equation*}
    \frac{1}{\vol(U_\D^{n})}\mathbbm{1}_{\alpha U_\D^n} \leftrightarrow \frac{1}{\vol(U_\J^{n})}\mathbbm{1}_{\iota(\alpha) U_\J^n},
\end{equation*}
as desired. Suppose now that $d(\alpha) = 0$. Then we can write $\alpha = t \alpha_0$ with $t \in F^\times$ and $\alpha_0 \in h_\D^{-1}(Z_u)$. Let $\iota \in \Emb(\o[\alpha_0]: \M)$. By Lemma \ref{lem:charpolyunramified} and Proposition \ref{prop:orbitalintegralneighbourhood} we have the matching
\begin{equation*}
\frac{1}{\vol(U_\D^n)}\mathbbm{1}_{\alpha_0 U_\D^n} \leftrightarrow \frac{1}{\vol(U_\M^{\lceil \frac{n}{2}\rceil})}\mathbbm{1}_{\iota(\alpha_0) U_\M^{\lceil \frac{n}{2}\rceil}}.
\end{equation*}
By equation \eqref{eq:transformationorbital} we see that 
\begin{equation*}
\frac{1}{\vol(U_\D^n)}\mathbbm{1}_{\alpha U_\D^n} \leftrightarrow \frac{1}{\vol(U_\M^{\lceil \frac{n}{2}\rceil})}\mathbbm{1}_{\iota(\alpha) U_\M^{\lceil \frac{n}{2}\rceil}},
\end{equation*}
as desired. The remaining case is $d(\alpha) = 2k$ with $k \geq 1$. In this case we can write $\alpha = t(1 + \varpi^k \alpha_0)$ with $t \in F^\times$ and $\alpha_0 \in h_\D^{-1}(Z_u)$. Let $\iota \in \Emb(\o[\alpha_0]; \O)$. Since $\alpha_0 \in \O_\D^\times$ and $\iota(\alpha_0) \in \M^\times$ we see that 
\begin{equation}\label{eq:transformationsetunramified}
    \alpha U_\D^n = t(1 + \varpi^k \alpha_0 U_\D^{n-2k}), \quad \iota(\alpha)U_\M^{\lceil \frac{n}{2}\rceil} = t(1 + \varpi^k \iota(\alpha_0)U_\M^{\lceil \frac{n}{2}\rceil - k}).
\end{equation}
By Lemma \ref{lem:charpolyunramified} and Proposition \ref{prop:orbitalintegralneighbourhood} we have the matching 
\begin{equation*}
    \frac{1}{\vol(U_\D^{n-2k})}\mathbbm{1}_{\alpha_0 U_\D^{n - 2k}} \leftrightarrow \frac{1}{\vol(U_\M^{\lceil \frac{n}{2}\rceil - k})}\mathbbm{1}_{\iota(\alpha_0) U_\M^{\lceil \frac{n}{2}\rceil - k}}.
\end{equation*} 
Finally, by equations \eqref{eq:transformationorbital} and \eqref{eq:transformationsetunramified} and Lemma \ref{lem:volumeprincipalsubgroups} we see that
\begin{equation*}
    \frac{1}{\vol(U_\D^{n})}\mathbbm{1}_{\alpha U_\D^{n}} \leftrightarrow \frac{1}{\vol(U_\M^{\lceil \frac{n}{2}\rceil})}\mathbbm{1}_{\iota(\alpha) U_\M^{\lceil \frac{n}{2}\rceil}},
\end{equation*} 
as desired.

\section{Global arguments.}\label{sec:globalarguments}
\subsection{Notation and background.}
In this section the context is global. We let $\Pl$ denote the set of different places of $\Q$. We can write $\Pl = \Primes \sqcup \{\infty\}$, where $\Primes$ is the set of nonarchimedean places of $\Q$ and $\infty$ is the only archimedean place. We also say that the places in $\Primes$ are the \emph{finite} places, while $\infty$ is the \emph{infinite} place of $\Q$. If $v \in \Pl$, we let $\Q_v$ be the completion of $\Q$ with respect to an absolute value $|\cdot|_v$ in the equivalence class of $v$. Thus
\begin{equation*}
    \Q_v = \begin{dcases}
        \Q_p, & \text{ if }v = (p) \in \Primes,\\
        \R, & \text{ if }v = \infty.
    \end{dcases}
\end{equation*}

Let $B$ be a quaternion algebra over $\Q$. We say that $B$ splits at $v \in \Pl$ if $B\otimes_\Q \Q_v \simeq M_2(\Q_v)$. Otherwise we say that $B$ is ramified at $v$. We let $\Ram(B) \subset \Pl$ be the set of ramified places. We know that $\Ram(B)$ is a finite set of even cardinality. Furthermore, $\Ram(B) = \varnothing$ if and only if $B \simeq M_2(\Q)$. The set $\Ram(B)$ is a complete invariant of $B$, in the sense that $\Ram(B_1) = \Ram(B_2)$ if and only if $B_1 \simeq B_2$ as $\Q$-algebras. See \cite[Main Theorem 14.6.1]{voight} for these assertions.

From now on we assume that $B$ is \emph{indefinite}, which means that $B$ splits at $\infty$. Therefore, $\Ram(B)$ consists of an even number of finite primes. We let $\O$ denote a maximal order on $B$. Since $B$ is indefinite, we know that $\O$ is unique up to conjugation by an element of $B^\times$ \cite[Corollary 28.5.10]{voight}.

Just as in the local case, $B$ comes equipped with the reduced trace $\trd: B \rightarrow \Q$ and reduced norm $\nrd: B\rightarrow \Q$ maps, as well as with the standard involution $x \mapsto \overline{x}:= \trd(x) - x$, with the usual properties. In particular, we note that $\O^\times = \O \cap \nrd^{-1}(\Z^\times)$. We let $B^1 := \ker(\nrd) = \{x \in B^\times : \nrd(x) = B^1\}$, and define $\O^1 := \O \cap B^1$. Since $B$ is indefinite, the isomorphism $B \otimes_\Q \R \simeq M_2(\R)$ induces an embedding $\iota_\infty: \O^1 \hookrightarrow \SL_2(\R)$. We know that $\iota(\O^1) \subset \SL_2(\R)$ is discrete and cofinite. Furthermore, $\iota(\O^1)\backslash \SL_2(\R)$ compact unless $B \simeq M_2(\Q)$ \cite[Theorem 38.1.3]{voight}.

As is customary, if $v \in \Pl$ we let $B_v := B \otimes_\Q \Q_v$. When $v$ is a finite place, we let $\O_v := \O \otimes_\Z \Z_v$. We also consider the adelic version of the previous objects. We use the notation $\A$ for the adeles over $\Q$, and $\A^\times$ for the ideles. We have restricted products
\begin{equation*}
    B_\A^\times := \sideset{}{'}\prod_{v \in \Pl}(B_v^\times, \O_v^\times), \quad B_\A^1 := \sideset{}{'}\prod_{v \in \Pl}(B_v^1, \O_v^1).
\end{equation*}
When the product is only over finite places we use the notation
\begin{equation*}
    \widehat{B}^\times := \sideset{}{'}\prod_{p \in \Primes}(B_p^\times, \O_p^\times), \quad \widehat{B}^1 := \sideset{}{'}\prod_{p \in \Primes}(B_p^1, \O_p^1).
\end{equation*}
We also let 
\begin{equation*}
    \widehat{\O}^\times := \prod_{p \in \Primes} \O_p^\times, \quad \widehat{\O}^1 := \prod_{p \in \Primes} \O_p^1,
\end{equation*}
which are maximal compact subgroups of $\widehat{B}^\times$ and $\widehat{B}^1$, respectively. For each $v$ we have a canonical embedding $\iota_v: B^\times_v \hookrightarrow B_\A^\times$, so that 
\begin{equation*}
    (\iota_v(x))_{w} = \begin{dcases}
        x, & \text{ if } v = w,\\
        1, & \text{ otherwise},
    \end{dcases}
\end{equation*}
and analogously for other adelic groups. We also have a \emph{diagonal embedding} $\iota_{\diag}: B^\times \hookrightarrow B_\A^\times$ such that $(\iota_{\diag}(x))_v = x$ for any $v \in \Pl$. If $B_v^\times$ or $B^\times$ are considered as subgroups of $B_\A^\times$ it will be via these embeddings, and similarly for other groups (e.g. for $B_\A^1$ or $E_\A^\times$ below). Thus, we may write $B_\infty^\times \subset B_\A^\times$ and $B^\times \subset B_\A^\times$ instead of $\iota_\infty^\times(B_\infty^\times) \subset B_\A^\times$ and $\iota_{\diag}(B^\times)\subset B_\A^\times$, respectively.

For a positive integer $q = \prod_p p^{n_p}$ we define principal congruence subgroups
\begin{equation*}
        \widehat{\O}^\times(q) := \prod_{p \in \Primes} \O_p^\times(p^{n_p}), \quad \widehat{\O}^1(q) := \prod_{p \in \Primes} \O_p^1(p^{n_p}),
\end{equation*}
where $\O_p^\times(p^{n_p}) = \O_p^\times \cap (1 + p^{n_p}\O_p)$ and $\O_p^1(p^{n_p}) = \O_p^1 \cap \O_p^\times(p^{n_p})$, as in Definition \ref{defi:subgroupsdefininglevel}. Globally we define $\O^1(q) := \O^1 \cap (1 + q\O)$. By strong approximation for $\O^1$ \cite[Corollary 28.5.14]{voight} we have isomorphisms
\begin{equation}\label{eq:chineseremainderidentification}
\begin{tikzcd}
(\O/q\O)^1 & \O^1(q)\backslash \O^1 \arrow[l, "\simeq"'] \arrow[r, "\simeq"] & \widehat{\O}^1(q)\backslash \widehat{\O}^1
\end{tikzcd}.
\end{equation}
The first arrow is induced by reduction modulo $q$, and the second by the diagonal embedding $\O^1\hookrightarrow \widehat{\O}^1$. Using this isomorphisms we obtain a natural bijection between functions $\varphi$ on $(\O/q\O)^1$ and functions $\widehat{\varphi}$ on $\widehat{\O}^1$ of level $q$, that is, which are $\widehat{\O}^1(q)$-invariant by left and right translation. By the Chinese remainder theorem 
\begin{equation*}
    (\O/q\O)^1 \simeq \prod_{p \mid q} (\O/p^{n_p}\O)^1.
\end{equation*}
We say that a function $\varphi$ on $(\O/q\O)^1$ is \emph{factorizable} if it factors as a product of functions on $(\O/p^{n_p}\O)^1$ via this isomorphism. Via the isomorphisms in \eqref{eq:chineseremainderidentification}, factorizable functions on $(\O/q\O)^1$ correspond to functions $\widehat{\varphi} = \prod_p \varphi_p$ obeying the formula $\widehat{\varphi}((x_p)_p) = \prod_p \varphi_p(x_p)$, where $\varphi_p$ is a function on $\O_p^1$ which is invariant under translation by elements of  $\O_p^1(p^{n_p})$.

Recall Definition \ref{defi:stablefunction}. Using the isomorphisms in \eqref{eq:chineseremainderidentification} we see that stable functions $\varphi$ on $(\O/q\O)^1$ correspond to functions $\widehat{\varphi}$ which are invariant under conjugation by $\widehat{\O}^\times$.

\subsubsection{Strong approximation and its consequences.}
Let $\widehat{\iota}$ denote the diagonal embedding $B^1 \hookrightarrow \widehat{B}^1$. This is the composition of $\iota_{\diag}: B^1 \hookrightarrow B_\A^1$ with the natural projection $B_\A^1 \rightarrow \widehat{B}^1$. 
\begin{lem}\label{lem:strongapproxB1general}
Let $B$ be an indefinite quaternion algebra over $\Q$. Then $\widehat{\iota}(B^1)$ is dense in $\widehat{B}^1$.
\end{lem}
\begin{proof}
This is a particular case of \cite[Theorem 28.5.3]{voight}.
\end{proof}

\begin{cor}\label{cor:strongapproxB1}
Let $B$ be an indefinite quaternion algebra over $\Q$. Then, we have 
\begin{equation*}
    B_\A^1 = B^1 B_\infty^1 \widehat{\O}^1.
\end{equation*}
\end{cor}
\begin{proof}
See \cite[Corollary 28.5.12]{voight}.
\end{proof}

\begin{lem}\label{lem:surjectivitynorm}
Let $B$ be an indefinite quaternion algebra over $\Q$. Then, we have $\nrd(B^\times) = \Q^\times$.    
\end{lem}
\begin{proof}
This is a particular case of \cite[Theorem 28.6.1]{voight}.
\end{proof}

Let $B_\infty^+ := \{x \in B_\infty^\times : \nrd(x) > 0\} \simeq \GL_2^+(\R)$.

\begin{lem}\label{lem:strongapproxBtimes}
Let $B$ be an indefinite quaternion algebra over $\Q$. Then, we have an equality
\begin{equation*}
    B_\A^\times = B^\times B_\infty^+\widehat{\O}^\times. 
\end{equation*}
\end{lem}
\begin{proof}
Clearly $\nrd(B_\infty^+) = \R_{> 0}$. We also know that $\nrd(B^\times) = \Q^\times$ (by Lemma \ref{lem:surjectivitynorm}) and that $\nrd(\widehat{\O}^\times) = \widehat{\Z}^\times$ (see \cite[Example 28.5.16]{voight}). By Corollary \ref{cor:strongapproxB1} we know that $B_\A^1 \subset B^\times B_\infty^+ \widehat{\O}^\times$. The proof of the lemma is completed by taking reduced norms and using that $\A^\times = \Q^\times\R_{> 0}\widehat{\Z}^\times$.
\end{proof}

\subsubsection{Quadratic separable algebras over $\Q$.}
Let $E$ a separable $\Q$-algebra of degree two. This means that either $E = \Q \times \Q$ (the split case), or else $E/\Q$ is a field extension of degree two. Just as in the local case, $E$ comes equipped with the reduced trace $\trd: B \rightarrow \Q$ and reduced norm $\nrd: E\rightarrow \Q$ maps, as well as with a standard involution $x \mapsto \overline{x}:= \trd(x) - x$ with the usual properties. In particular, any $x \in E$ satisfies its characteristic polynomial, so that $x^2 - \trd(x)x + \nrd(x) = 0$. When $E$ is a field, then $x \mapsto \overline{x}$ is just the nontrivial Galois automorphism of the extension $E/F$. 

We let $\Emb(E; B)$ denote the set of $\Q$-algebra embeddings of $E$ in $B$. Any embedding $\iota \in \Emb(E; B)$ is compatible with $\trd$ and $\nrd$, so that $\trd(\iota(x)) = \trd(x)$ and $\nrd(\iota(x)) = \nrd(x)$. We let $E^1 := \{x \in E: \nrd(x) = 1\}$. Clearly, for any $\iota \in \Emb(E; B)$ we have $E^1 = \iota^{-1}(B^1)$. 

As a particular case of the Skolem-Noether theorem, we know that $B$ acts transitively on $\Emb(E; B)$ by conjugation \cite[Main Theorem 7.7.1]{voight}. We also know that
\begin{equation}
    \Emb(E; B) \neq \varnothing \iff \Emb(E_v; B_v) \neq \varnothing \text{ for all } v \in \Pl \iff E_v \text{ is a field for all } v \in \Ram(B),
\end{equation}
see \cite[Proposition 14.6.7]{voight}. There is a unique maximal order in $E$, which we denote by $\o_E$. If $p$ is a finite prime, then we have $\o_{E_p} = (\o_E)_p$. That is, the completion of the maximal order of $E$ with respect to $p$ is the maximal order of $E_p$. 

We let $\o_E^1 = \o_E \cap E^1$. We use the notations $E_\A^\times$, $E_\A^1$, $\widehat{E}^\times$, $\widehat{E}^1$, $\widehat{\o_E}^\times$ and $\widehat{\o_E}^1$ similarly as above. In particular, we have
\begin{equation*}
    E_\A^\times := \sideset{}{'}\prod_{v \in \Pl}(E_v^\times, \o_{E_v}^\times), \quad E_\A^1 := \sideset{}{'}\prod_{v \in \Pl}(E_v^1, \o_{E_v}^1).
\end{equation*}

\subsubsection{Evaluation of some integrals.}\label{subsubsec:haarmeasureB1}
The Haar measure on $B_\A^1$ is a product of Haar measures on $B_\infty^1$ and $\widehat{B}^1$. The Haar measure on $\widehat{B}^1$ is chosen so that $\vol(\widehat{\O}^1) = 1$. Note that $\iota_{\diag}(B^1)\cap (B^1_\infty \widehat{\O}^1) = \iota_{\diag}(\O^1)$. By Corollary \ref{cor:strongapproxB1} the inclusion $B_\infty^1 \widehat{\O}^1 \hookrightarrow B_\A^1$ induces a bijection
\begin{equation}\label{eq:identificationofhomogeneouspaces}
\begin{tikzcd}
\iota_{\diag}(\O^1)\backslash (B^1_\infty \cdot \widehat{\O}^1) \arrow[rr, "\simeq"] &  & \iota_{\diag}(B^1)\backslash B_\A^1
\end{tikzcd},
\end{equation}
which can be seen to be a homeomorphism. 
\begin{lem}\label{lem:computationadelicintegralB1}
Let $f$ be a measurable function on $B_\A^1$ which satisfies
\begin{enumerate}
    \item [i)] $f(\iota_{\diag}(x)g) = f(g)$ for $x \in B^1$, 
    \item [ii)] $f(gu)$ for $u \in \widehat{\O}^1$. 
\end{enumerate}
Then, we have $f \in L^1(\iota_{\diag}(B^1)\backslash B_\A^1)$ if and only if $f \circ \iota_\infty \in L^1(\iota_\infty(\O^1) \backslash \SL_2(\R))$. Furthermore, in case $f \in L^1(\iota_{\diag}(B^1)\backslash B_\A^1)$ we have
\begin{equation}\label{eq:integralshomogeneousspacesB1}
    \int_{\iota_{\diag}(B^1)\backslash B_\A^1} f(g) \, dg = \int_{\iota_\infty(\O^1) \backslash \SL_2(\R)} f(\iota_\infty(g)) \, dg.
\end{equation}
\end{lem}

\begin{proof}
This follows from the homeomorphism \eqref{eq:identificationofhomogeneouspaces} and the fact that $\vol(\widehat{\O}^1) = 1$. 
\end{proof}
We consider the measure $\frac{dy}{|y|}$ on $\R^\times$, and choose a measure on $B_\infty^\times$ compatible with the Haar measures already fixed on $B^1_\infty$ and $\R^\times$ and with the exact sequence
\begin{equation*}
\begin{tikzcd}
1 \arrow[r] & B^1_\infty \arrow[r] & B^\times_\infty \arrow[r, "\nrd"] & \R^\times \arrow[r] & 1.
\end{tikzcd}
\end{equation*}
Let $Z_\infty^+ = \{\begin{psmallmatrix}
    t^{1/2} & 0\\
    0 & t^{1/2}
\end{psmallmatrix}: t \in \R_{> 0}\} \simeq \R_{> 0}$, and equip this group with the Haar measure $dt/t$. If we identify $B_\infty^\times \simeq \GL_2(\R)$, then the choices of Haar measures are compatible with the isomorphism $B_\infty^+ \simeq Z_\infty^+ \times B_\infty^1$. This means that, for $f \in C_c^\infty(\GL_2^+(\R))$, we have 
\begin{equation}\label{eq:haarmeasuregl2sl2}
    \int_{g \in \GL_2^+(\R)} f(g) \, dg = \int_{t \in \R_{> 0}}\int_{h \in \SL_2(\R)} f\left(\begin{psmallmatrix}
    t^{1/2} & 0\\
    0 & t^{1/2}
\end{psmallmatrix}h\right)\, \frac{dt}{t}\, dh.
\end{equation}  
The Haar measure on $\widehat{B}^\times$ is chosen so that $\vol(\widehat{\O}^\times) = 1$, and the Haar measure on $B_\A^\times$ is the product of the Haar measures already chosen for $B_\infty^\times$ and $\widehat{B}^\times$. Note that 
\begin{equation*}
    Z_\infty^+ \iota_{\diag}(B^\times) \cap B_\infty^+ \widehat{\O}^\times = Z_\infty^+ \iota_{\diag}(\O^1).
\end{equation*}
It follows from Lemma \ref{lem:strongapproxBtimes} that the inclusion $B_\infty^+ \widehat{\O}^\times \hookrightarrow B_\A^\times$ induces a bijection
\begin{equation}\label{auxeq:homogeneousspaceBtimes}
\begin{tikzcd}
Z_\infty^+\iota_{\diag}(\O^1) \backslash (B_\infty^+ \widehat{\O}^\times) \arrow[rr, "\simeq"] &  & Z_\infty^+\iota_{\diag}(B^\times) \backslash B_\A^\times,
\end{tikzcd}
\end{equation}
which can be seen to be an homeomorphism.

\begin{lem}\label{lem:computationadelicintegralBtimes}
    Let $f$ be a measurable function on $B_\A^\times$ which satisfies
\begin{enumerate}
    \item [i)] $f(\iota_{\diag}(x)g) = f(g)$ for $x \in B^\times$.
    \item [ii)] $f(zg) = f(g)$ for $z \in Z_\infty^+$.
    \item [iii)] $f(gu)$ for $u \in \widehat{\O}^\times$. 
\end{enumerate}
Then, $f \in L^1(Z_\infty^+\iota_{\diag}(B^\times)\backslash B_\A^\times)$ if and only if $f \circ \iota_\infty \in L^1(\iota_\infty(\O^1) \backslash \SL_2(\R))$. Furthermore, in this case we have the equality
\begin{equation}\label{eq:integralshomogeneousspaces}
    \int_{Z_\infty^+ \iota_{\diag}(B^\times)\backslash B_\A^\times} f(g) \, dg = \int_{\iota_{\infty}(\O^1)\backslash \SL_2(\R)} f(\iota_\infty(g)) \, dg.
\end{equation}
\end{lem}

\begin{proof}
This follows from the homeomorphism \eqref{auxeq:homogeneousspaceBtimes}, the fact that $\vol(\widehat{\O}^\times) = 1$, the fact that $\iota_\infty(\O^1)\cap Z_\infty^+ = \{1\}$, and from our choice of measures for $B_\infty^+$, $Z_\infty^+$ and $B_\infty^1$.
\end{proof}

\subsection{Classical-adelic comparison.}
\subsubsection{Classical geodesic sums.}
Recall that $B$ is an indefinite quaternion algebra over $\Q$, possibly split, and that $\O$ is a maximal order in $B$. Given $t > 2$ we define $N(t) > 1$ implicitly by the equation $N(t)^{1/2} + N(t)^{-1/2} = t$. Given $P \in \O^1$ with $\trd(P) > 2$ we let $C_{\O^1}(P) = \{\gamma \in \O^1 : \gamma^{-1}P\gamma = P\}$ be its centralizer in $\O^1$. We let $\nu(P) := [C_{\O^1}(P): \{\pm I\}\langle P\rangle]$. Let $\varphi$ be a class function on $(\O/q\O)^1$. Define the sums
\begin{equation}\label{eq:chebyshevfixedtrace}
    \psi_{\O^1}(t; \varphi) := \log N(t)\sum_{\substack{\{P\}_{\O^1}\\
    \trd(P) = t}} \frac{\varphi(P)}{\nu(P)},
\end{equation}
which are related to the sums \eqref{eq:chebyshevquaternion} considered in the introduction by the formula
\begin{equation}\label{eq:chebyshevsumoftraces}
    \Psi_{\O^1}(X; \varphi) = \sum_{2 < t \leq X^{1/2} + X^{-1/2}} \psi_{\O^1}(t; \varphi),
\end{equation}
where $X \geq 1$. We also write 
\begin{equation}\label{eq:untwistedchebyshevfixedtrace}
    \psi_{\O^1}(t) := \log N(t)\sum_{\substack{\{P\}_{\O^1}\\
    \trd(P) = t}} \frac{1}{\nu(P)},
\end{equation}
so that 
\begin{equation*}
    \Psi_{\O^1}(X) = \sum_{2 < t \leq X^{1/2} + X^{-1/2}} \psi_{\O^1}(t).
\end{equation*}
Note that $\psi_{\O^1}(t) = \psi_{\O^1}(t; \mathbbm{1}_{(\O/q\O)^1})$ for any $q \geq 1$.

\subsubsection{Adelic geodesic sums for $B^1$.}
Let $\varphi_\A \in C_c^\infty(B_\A^1)$. Given $t > 2$ we define the function $K_t$ on $B_\A^\times$ by
\begin{equation}\label{eq:defitracekernel}
    K_t(g) := \sum_{\substack{\gamma \in B^1\\\trd(\gamma) = t}} \varphi_\A(g^{-1}\gamma g).
\end{equation}
Note that $K_t(\delta g) = K_t(g)$ for $\delta \in B^\times$. Thus, it makes sense to integrate this function over $B^1\backslash B_\A^1$, and we denote the resulting integral by  
\begin{equation}\label{eq:adelicdefinitionhyperbolicterm}
    \Lambda_{B^1}(t; \varphi_\A) := \int_{B^1\backslash B^1_\A} K_t(g) \, dg. 
\end{equation}
Suppose we are given a class function $\varphi$ on $(\O/q\O)^1$. It corresponds, via the isomorphism \eqref{eq:chineseremainderidentification}, to a class function on $\widehat{\O}^1$, denoted by $\widehat{\varphi}$, which is invariant under translations by $\widehat{\O}^1(q)$. We can extend $\widehat{\varphi}$ to a function on $\widehat{B}^1$ by letting it vanish outside of $\widehat{\O}^1$. We let $\varphi_\infty \in C_c^\infty(B_\infty^1)$ be arbitrary and define $\varphi_\A \in C_c^\infty(B_\A^1)$ as the product of $\varphi_\infty$ and $\widehat{\varphi}$, so that 
\begin{equation}\label{eq:defiadelicfunction}
    \varphi_\A(\iota_\infty(\gamma_\infty)\widehat{\gamma}) := \varphi_\infty(\gamma_\infty)\widehat{\varphi}(\widehat{\gamma}), \quad \text{where }\gamma_\infty \in B^1_\infty, \widehat{\gamma} \in \widehat{B}^1.
\end{equation}
If we define $K_t$ as in \eqref{eq:defitracekernel} using the function $\varphi_\A$ from \eqref{eq:defiadelicfunction}, then we have $K_t(gu) = K_t(g)$ for $u \in \widehat{\O}^1$. Thus, if we apply Lemma \ref{lem:computationadelicintegralB1} to \eqref{eq:adelicdefinitionhyperbolicterm} we deduce that 
\begin{equation*}
    \Lambda_{B^1}(t; \varphi_\A) = \int_{\O^1 \backslash \SL_2(\R)} K_t(\iota_\infty(g)) \, dg. 
\end{equation*}
Since $B_\infty^1$ commutes with $\widehat{B}^1$, we get  
\begin{equation*}
    \varphi_\A(\iota_\infty(g)^{-1}\gamma \iota_\infty(g)) = \begin{dcases}
        \varphi_\infty(g^{-1}\gamma g) \varphi(\gamma \modulo q), & \text{ if }\gamma \in \O^1,\\
        0, & \text{ otherwise}.
    \end{dcases}
\end{equation*}
Therefore, we deduce that 
\begin{equation*}
\Lambda_{B^1}(t; \varphi_\A) = \int_{\O^1 \backslash \SL_2(\R)} \sum_{\substack{P \in \O^1\\
\trd(P) = t}} \varphi_\infty(g^{-1}P g) \varphi(P \modulo q)\, dg.
\end{equation*}
We can break up the inner sum into $\O^1$-conjugacy classes 
\begin{equation*}
\sum_{\substack{P \in \O^1\\
\trd(P) = t}} \varphi_\infty(g^{-1}P g) \varphi(P \modulo q) = \sum_{\substack{\{P\}_{\O^1}\\
    \trd(P) = t}} \varphi(P \modulo q) \sum_{\delta \in C_{\O^1}(P)\backslash \O^1} \varphi_\infty(g^{-1}\delta^{-1}P \delta g),
\end{equation*}
and after exchanging the outer sum with the integral and applying Fubini's theorem we obtain
\begin{equation*}
    \Lambda_{B^1}(t; \varphi_\A) = \sum_{\substack{\{P\}_{\O^1}\\
    \trd(P) = t}} \varphi(P \modulo q) \int_{C_{\O^1}(P)\backslash \SL_2(\R)} \varphi_\infty(g^{-1}P g)\, dg.
\end{equation*}
For each $P\in \SL_2(\R)$ with $\trd(P) = t > 2$,  there is $g \in \SL_2(\R)$ such that $g^{-1}Pg = \begin{psmallmatrix}
    N(t)^{\frac{1}{2}} & 0\\
    0 & N(t)^{-\frac{1}{2}}
\end{psmallmatrix}$ and $g^{-1}C_{\O^1}(P)g = \{\pm \begin{psmallmatrix}
    N(t)^{\frac{k}{2\nu(P)}} & 0\\
    0 & N(t)^{\frac{k}{2\nu(P)}}\end{psmallmatrix}: k \in \Z\}$. Let $A = \R \times \R$, which we embed in $M_2(\R)$ via $(y_1, y_2)\mapsto \begin{psmallmatrix}
    y_1 & 0\\
    0 & y_2
    \end{psmallmatrix}$. We choose the Haar measure on $A^1$ so that $\vol_{A^1}(\{(y^{1/2}, y^{-1/2}): y \in [1, T]\} = \log T$. In particular,  we have $\vol(C_{\O^1}(P) \backslash \iota(E^1)) = (\log N(t))/\nu(P)$. There is a unique Haar measure on $A^\times$ which is compatible with the exact sequence 
    \begin{equation*}
\begin{tikzcd}
1 \arrow[r] & A^1 \arrow[r] & A^\times \arrow[r, "\nrd"] & \R^\times \arrow[r] & 1,
\end{tikzcd}
    \end{equation*}
where $\R^\times$ has the standard Haar measure $dy/|y|$. 
Note that the inclusion $\SL_2(\R) \subset \GL_2(\R)$ induces a measure-preserving homeomorphism $\iota(E^1)\backslash \SL_2(\R) \simeq \iota(E^\times)\backslash \GL_2(\R)$, recall Section \ref{subsubsec:haarmeasureB1} for the choice of measures. We deduce that 
\begin{equation*}
\begin{aligned}
    \int_{C_{\O^1}(P)\backslash \SL_2(\R)} \varphi_\infty(g^{-1}P g)\, dg  = \frac{\log N(t)}{\nu(P)} \O_{\GL_2(\R)}(\varphi_\infty; \R \times \R, (N(t)^{1/2}, N(t)^{-1/2})),
\end{aligned}
\end{equation*}
where the definition of the local orbital integral in the archimedean case is the same as in the nonarchimedean case, see \eqref{eq:defioforbitalintegral}.
For brevity, we write $Q_{\varphi_\infty} \in C_c(\R_{> 0})$ for the function 
\begin{equation}\label{eq:defiabeltransform}
    Q_{\varphi_\infty}(y) :=  \O_{\GL_2(\R)}(\varphi_\infty; \R \times \R, (y^{1/2}, y^{-1/2})),
\end{equation}
We record what we have shown so far. 
\begin{lem}\label{lem:adelicclassicalcomparisonB1fixedt}
Let $\varphi$ be a class function on $(\O/q\O)^1$ and $\varphi_\infty \in C_c^\infty(B_\infty^1)$. Let $\varphi_\A \in C_c^\infty(B_\A^1)$ be defined by \eqref{eq:defiadelicfunction}. Then, for any positive integer $t > 2$ we have 
\begin{equation*}
    \Lambda_{B^1}(t; \varphi_\A) = Q_{\varphi_\infty}(N(t))\psi_{\O^1}(t; \varphi),
\end{equation*}
where $Q_{\varphi_\infty}$ is defined in \eqref{eq:defiabeltransform}.
\end{lem}

\subsubsection{Adelic geodesic sums for $B^\times$.}
Given any $\varphi_\A \in C_c^\infty(B_\A^1)$, the function $K_t$ defined on $B_\A^\times$ by equation \eqref{eq:defitracekernel} satisfies $K_t(zg) = K_t(g)$ for any $z$ in the center of $B_\A^\times$, in particular for $z \in Z_\infty^+$. Thus, the following integral is well-defined
\begin{equation}\label{eq:adelicdefinitionhyperbolictermBtimes}
    \Lambda_{B^\times}(t; \varphi_\A) := \int_{Z_\infty^+ B^\times\backslash B_\A^\times} K_t(g) \, dg = \int_{Z_\infty^+ B^\times \backslash B_\A^\times}  \sum_{\substack{\gamma \in B^1\\\trd(\gamma) = t}} \varphi_\A(g^{-1}\gamma g) \, dg. 
\end{equation}
Recall Definition \ref{defi:stablefunction}.
If we define $\varphi_\A$ as in equation \eqref{eq:defiadelicfunction}, starting with a stable function $\varphi$, then $\varphi_\A$ is invariant under conjugation by $\widehat{\O}^\times$. Thus, we have $K_t(gu)  = K_t(g)$ for $u \in \widehat{\O}^\times$. If we apply Lemma \ref{lem:computationadelicintegralBtimes} and proceed as above we obtain the following variant of Lemma \ref{lem:adelicclassicalcomparisonB1fixedt}. 

\begin{lem}\label{lem:adelicclassicalcomparisonBtimesfixedt}
Let the notation be as in Lemma \ref{lem:adelicclassicalcomparisonB1fixedt}, and assume that $\varphi$ is stable. Then, we have
\begin{equation*}
    \Lambda_{B^\times}(t; \varphi_\A) = \Lambda_{B^1}(t; \varphi_\A) = Q_{\varphi_\infty}(N(t))\psi_{\O^1}(t; \varphi).
\end{equation*}
\end{lem}

Let us look more closely at the definition of $\Lambda_{B^\times}(t; \varphi_\A)$ in \eqref{eq:adelicdefinitionhyperbolictermBtimes}. By the Noether--Skolem theorem we know that all the elements $\gamma \in B^1$ such that $\trd(\gamma) = t$ are conjugate under $B^\times$. Consider the field $E := F[X]/(X^2 - tX + 1)$ and the element $x \in E$ such that $x = X \modulo (X^2 - tX + 1)$. If $\Emb(E; B) = \varnothing$, then $\Lambda_{B^\times}(t; \varphi) = 0$, since the inner sum is empty. Suppose that $\Emb(E; B)$ is nonempty, and let $\iota \in \Emb(E; B)$. Then the map $g \mapsto g^{-1}\iota g$ induces a bijection
\begin{equation*}
    \iota(E^\times)\backslash B^\times \longleftrightarrow \{\gamma \in B^1: \trd(\gamma) = t\}.
\end{equation*} 
Also, note that the centralizer of $\iota(x)$ in $B_\A^\times$ is $\iota(E_\A^\times)$. We deduce that 
\begin{equation*}
    \Lambda_{B^\times}(t; \varphi_\A) = \begin{dcases}
    0, & \text{ if }\Emb(E; B) = \varnothing,\\
    \vol(Z_\infty^+ E^\times \backslash E_\A^\times) \cdot \int_{\iota(E_\A^\times)\backslash B_\A^\times} \varphi_\A(g^{-1}\iota(x)g) \, dg, & \text{ otherwise}.\end{dcases} 
\end{equation*}
If $\varphi_\A$ is constructed as in \eqref{eq:defiadelicfunction} then we have
\begin{equation*}
    \int_{\iota(E_\A^\times)\backslash B_\A^\times} \varphi(g^{-1}\iota(x)g) \, dg = \O_{B_\infty^\times}(\varphi_\infty; E_\infty, x) \cdot \O_{\widehat{B}^\times}(\widehat{\varphi}; E, x)
\end{equation*}
where in general we define
\begin{equation}\label{eq:defifinitepartorbitalintegral}
\O_{\widehat{B}^\times}(\widehat{\varphi}; E, x) := \begin{dcases}
    0, & \text{ if }\Emb(E; B) = \varnothing,\\
    \int_{\iota(\widehat{E}^\times)\backslash \widehat{B}^\times} \widehat{\varphi}(g^{-1}\iota(x)g) \, dg, & \text{ otherwise}.
\end{dcases}
\end{equation}
Since $t > 2$ we have $E_\infty \simeq \R \times \R = A$, with $x$ corresponding to $(N(t)^{1/2}, N(t)^{-1/2})$, as in the discussion leading up to Lemma \ref{lem:adelicclassicalcomparisonB1fixedt}. Choosing the same Haar measure on $A^\times$ as we did earlier one can check that
\begin{equation*}
    \vol(Z_\infty^+ E^\times \backslash E_\A^\times) = 2\reg(E) h(E) \vol_{\widehat{E}^\times}(\widehat{\o_{E}}^\times),
\end{equation*}
where $\reg(E)$ is the regulator of $E$, $\o_{E}$ its ring of integers, and $h(E) = |E^\times \backslash \widehat{E}^\times/\widehat{\o_E}^\times|$ is the class number of $E$. If we compare with Lemma \ref{lem:adelicclassicalcomparisonB1fixedt} we obtain the following identity, which after summing over $t > 2$ generalizes \eqref{eq:connectiontoclassnumbers}. 

\begin{prop}\label{prop:expressionchebyshevprimegeodesic}
Let $\varphi$ be a stable class function on $(\O/q\O)^1$. For $t> 2$, let $E := \Q[X]/(X^2 - tX + 1)$ and $x := X \modulo (X^2 - tX + 1)$.
Then we have
\begin{equation*}
\psi_{\O^1}(t; \varphi) =
    2\reg(E) h(E) \vol_{\widehat{E}^\times}(\widehat{\o_{E}}^\times)\O_{\widehat{B}^\times}(\widehat{\varphi}; E, x).
\end{equation*}
\end{prop}
In particular, if $\varphi$ is factorizable, so that $\widehat{\varphi} = \prod_p \varphi_p$, then we have
\begin{equation*}
    \psi_{\O^1}(t; \varphi) = 2\reg(E) h(E) \vol_{\widehat{E}^\times}(\widehat{\o_{E}}^\times)\prod_p \O_{B_p^\times}(\varphi_p; E_p, x).
\end{equation*}
Recall here the convention that a local orbital integral $\O_{B_p^\times}(\varphi_p; E_p, x)$ vanishes if $\Emb(E_p; B_p) = \varnothing$. 

\subsection{Proof of Proposition \ref{prop:classicalmatchingforstable}.}\label{subsec:proofofclassicalmatching}
Recall the notation in the statement of Proposition \ref{prop:classicalmatchingforstable}.
By linearity we can assume that $\varphi$ is factorizable. Write $q = \prod_{p}p^{n_p}$. The factorizable function $\varphi$ on $(\O/q\O)^1$ corresponds, via the isomorphisms in \eqref{eq:chineseremainderidentification}, to a function $\widehat{\varphi} \in C_c^\infty(\widehat{\D}^1)$, supported on $\widehat{\O}^1$, which factorizes as $\widehat{\varphi} = \prod_p \varphi_p$, with each component $\varphi_p \in C_c^\infty(\D^1_p)$ supported on $\O_p^1$ and of level $p^{n_p}$. We construct a corresponding function $\widehat{f} \in C_c^\infty(\SL_2(\widehat{\Q}))$ as a product $\widehat{f} := \prod_p f_p$, where the component $f_p \in C_c^\infty(\SL_2(\Q_p))$ is given as follows:
\begin{itemize}
    \item If $p \not \in \Ram(\D)$ there exists a $\Q_p$-algebra isomorphism $r_p: \D_p \rightarrow M_2(\Q_p)$ such that $r_p(\O_p) = M_2(\Z_p)$. In this case we let $f_p := \varphi_p \circ r_p^{-1}$. 
    \item If $p \in \Ram(\D)$, we choose $f_p$ so that $\varphi_p \leftrightarrow f_p$ in the sense of Definition \ref{defi:matchingorbital}. By Proposition \ref{prop:explicitmatching} we can take $f_p \in C_c^\infty(\SL_2(\Z_p))$ of level $p^{n_p + 1}$ and invariant under conjugation by $\GL_2(\Z_p)$.
\end{itemize}
Note that $\widehat{f}$ is supported on $\SL_2(\widehat{\Z})$, is invariant under conjugation by $\GL_2(\widehat{\Z})$, and is of level $q' := \disc(\D)q$. Thus, via the isomorphisms in \eqref{eq:chineseremainderidentification}, the function $\widehat{f}$ corresponds to some stable function $f$ on $\SL_2(\Z/q'\Z)$. Let $t > 2$ be arbitrary, and let $E$ and $x$ be as in Proposition \ref{prop:expressionchebyshevprimegeodesic}. 
Note that Proposition \ref{prop:expressionchebyshevprimegeodesic} applies to $\D^\times$ as well as to $\GL_2$. The choices of test functions guarantee that   
\begin{equation*}
    \O_{\D_p^\times}(\varphi_p; E_p, x) = \O_{\GL_2(\Q_p)}(f_p; E_p, x)
\end{equation*}
for any finite prime $p$. Thus, applying Proposition \ref{prop:expressionchebyshevprimegeodesic} twice we deduce that  
\begin{equation*}
\begin{aligned}
\psi_{\O^1}(t; \varphi) & = 2 \reg(E) h(E) \vol_{\widehat{E}^\times}(\widehat{\o_{E}}^\times)\prod_p \O_{\D_p^\times}(\varphi_p; E_p, x)\\
&  = 2 \reg(E) h(E) \vol_{\widehat{E}^\times}(\widehat{\o_{E}}^\times)\prod_p \O_{\GL_2(\Q_p)}(f_p; E_p, x) \\
& = \psi_{\SL_2(\Z)}(t; f).
\end{aligned}
\end{equation*}
Since this equality holds for any $t > 2$, we deduce that $\Psi_{\O^1}(X; \varphi) = \Psi_{\SL_2(\Z)}(X; f)$ for any $X > 1$, as desired.  

\subsection{Almost stability.}\label{subsec:almoststability}
Consider $\varphi_\A \in C_c^\infty(B_\A^1)$. We know that $\varphi_\A$ is a finite linear combination of functions of the form $\prod_v \varphi_v$, where $\varphi_v \in C_c^\infty(B_v^1)$ for each $v \in \Pl$ and $\varphi_p = \mathbbm{1}_{\SL_2(\Z_p)}$ for almost all primes $p$. Define 
\begin{equation*}
    M = S_{B_\A^\times}(\varphi_\A) := \{g \in B_\A^\times : \varphi_\A^g = \varphi_\A\}.
\end{equation*}
If $S \subset \Pl$ is a finite set of places containing $\infty$, we define 
\begin{equation*}
\widehat{\O}^{S, \times} := \prod_{p \notin S} \O_p^\times, \quad \widehat{\Z}^{S, \times} := \prod_{p \notin S} \Z_p^\times.
\end{equation*}
By the description of $\varphi_\A$ we know that there is a finite set $\{\infty\}\subset S \subset \Pl$ such that $\widehat{\O}^{S, \times} \subset M$. Note that $S$ depends on $\varphi_\A$. We also define
\begin{equation}\label{eq:stabilizerorbitalintegral}
    H = S_{B_\A^\times}(\Lambda_{B^1}(t; \varphi_\A)) :=  \{g \in B_\A^\times : \Lambda_{B^1}(t; \varphi_\A^g) = \Lambda_{B^1}(t; \varphi_\A)\}.
\end{equation}
Clearly $M \subset  H$. Proposition \ref{prop:almoststabilityclassical} will follow quickly from the following result. 

\begin{prop}\label{prop:stability}
Let $t > 2$ and consider the field $E := \Q[X]/(X^2 - tX + 1)$. If $\Ram(E) \nsubseteq S$, then $H = B_\A^\times$.
\end{prop}

Before giving the proof of Proposition \ref{prop:stability} we need to collect a technical lemma. 

\begin{lem}\label{lem:subgroupequaltoideles}
Let $E/\Q$ be a quadratic extension and let $S \subset \Pl$ be a finite set contaning $\infty$. Then $\Ram(E) \nsubseteq S$ if and only if $\Q^\times \nrd(E_\A^\times) \widehat{\Z}^{S, \times} = \A^\times$.
\end{lem}
\begin{proof}
The result follows from Artin reciprocity, see \cite[Chapters VI and VII]{casselsfrohlich} for a nice exposition of class field theory. Recall that for a place $v \in \Pl$ we have $E_v := E \otimes_\Q \Q_v$, which is a separable algebra of degree two over $\Q_v$. If $E_v$ is not a field, then $E_v \simeq \Q_v \times \Q_v$, thus we have $E^\times \simeq \Q_v^\times \times \Q_v^\times$ and $\nrd(E_v^\times) = \Q_v^\times$. 

Let $p \in \Primes$, and suppose that $E_p$ is a field. We have a commutative diagram 
\begin{equation*}
\begin{tikzcd}
\Q_p^\times \arrow[rr, "\psi_p"] \arrow[d, hook, "\iota_p"] &  & \Gal(E_p/\Q_p) \arrow[d, "\simeq"] \\
\A^\times \arrow[rr, "\psi"]                     &  & \Gal(E/\Q)                    
\end{tikzcd}
\end{equation*}
where $\psi$ is the global Artin map and $\psi_p$ is the local Artin map. The maps $\psi_p$, $\psi$ are surjective with kernels 
\begin{equation*}
    \ker \psi_p = \nrd(E_p^\times), \quad \ker \psi = \Q^\times \nrd(E_\A^\times).
\end{equation*}
In particular, we have $[\A^\times: \Q^\times \nrd(E_\A^\times)] = 2$. Furthermore, locally we know that $\Z_p^\times \subseteq \nrd(E_p^\times)$ if and only if $E_p$ is unramified. Thus, if $p \in \Ram(E)$ we have $\Z_p^\times \nsubseteq \ker \psi_p$. 

If $\Ram(E) \nsubseteq S$ we deduce that   $\widehat{\Z}^{S, \times} \nsubseteq \ker \psi = \Q^\times \nrd(E_\A^\times)$. Thus, in this case the group $\Q^\times \nrd(E_\A^\times) \widehat{\Z}^{S, \times}$ is strictly larger than $\Q^\times \nrd(E_\A^\times)$. Since $[\A^\times: \Q^\times \nrd(E_\A^\times)] = 2$, the only option is that $\Q^\times \nrd(E_\A^\times) \widehat{\Z}^{S, \times} = \A^\times$, as desired.

On the other hand, if $\Ram(E) \subseteq S$, then $\Z_p^\times \subseteq \nrd(E_p^\times) = \ker \psi_p$ for all $p \in \Primes - S$. It follows that $\widehat{\Z}^{S, \times} \subseteq \ker \psi = \Q^\times \nrd(E_\A^\times)$. Thus, $[\A^\times: \Q^\times \nrd(E_\A^\times)\widehat{\Z}^{S, \times}] = 2$, and we have $ \Q^\times \nrd(E_\A^\times)\widehat{\Z}^{S, \times} \neq \A^\times$, as desired.
\end{proof}

\begin{proof}[Proof of Proposition \ref{prop:stability}.]
We start by rewriting $\Lambda_{B^1}(t; \varphi)$ as a sum of orbital integrals, recall \eqref{eq:adelicdefinitionhyperbolicterm}. Given $\gamma \in B^1$, let $C(\gamma)$ be its centralizer in $B^1$. If we group the elements of $B^1$ by conjugacy classes we see that
\begin{equation}\label{eq:rewritingB1contributionstep1}
    \Lambda_{B^1}(t; \varphi) = \sum_{\substack{\{\gamma\}_{B^1}\\
    \trd(\gamma)= t}} \int_{C(\gamma)\backslash B^1_\A}\varphi(g^{-1}\gamma g)\, dg.
\end{equation}
We let $x := X \modulo (X^2 - tX + 1) \in E$. We have a bijection 
\begin{equation*}
    \Emb(E; B) \longleftrightarrow \{\gamma \in B^1: \trd(\gamma) = t\}
\end{equation*}
which assigns to $\iota \in \Emb(E; B)$ the element $\iota(x) \in B^1$. The bijection respects the action by conjugation of $B^1$ on each side, so taking equivalence classes we obtain
\begin{equation*}
    \Emb(E; B)/B^1 \longleftrightarrow \{\{\gamma\}_{B^1}: \trd(\gamma) = t\}.
\end{equation*}
It follows that we can rewrite equation \eqref{eq:rewritingB1contributionstep1} as 
\begin{equation}\label{eq:rewritingB1contributionstep2}
    \Lambda_{B^1}(t; \varphi) = \vol(E^1 \backslash E^1_\A) \sum_{\iota \in \Emb(E, B)/B^1} \int_{\iota(E_\A^1)\backslash B_\A^1}\varphi(g^{-1}\iota(x)g) \, dg.
\end{equation}
Let us introduce the notation
\begin{equation}\label{eq:defiglobalorbitalB1}
    \O^1(\varphi; E, x, \iota) := \int_{\iota(E^1_\A)\backslash B_\A^1}\varphi(g^{-1}\iota(x)g)\, dg,
\end{equation}
so that equation \eqref{eq:rewritingB1contributionstep2} can be rewritten as
\begin{equation}\label{eq:rewritingB1contributionstep3}
    \Lambda_{B^1}(t; \varphi)  = \vol(E^1 \backslash E^1_\A) \sum_{\iota \in \Emb(E, B)/B^1} \O^1(\varphi; E, x, \iota).
\end{equation}
We collect the following observations about these orbital integrals:
\begin{enumerate}
    \item [i)] If $h \in B^\times$ then $\O^1(\varphi; E, x, \iota) = \O^1(\varphi^h; E, x, \prescript{h}{}\iota)$.
    \item [ii)] If $h \in \iota(E_\A^\times) \subset B_\A^\times$ then $\O^1(\varphi^h; E, x, \iota) = \O^1(\varphi; E, x, \iota)$.
    \item [iii)] If $h \in B^1_\A$ then $\O^1(\varphi^h; E, x, \iota) = \O^1(\varphi; E, x, \iota)$.
\end{enumerate}
Parts i) and ii) follow by applying the change of variables $g \mapsto h^{-1}gh$ in the integral of \eqref{eq:defiglobalorbitalB1}. In part ii) we also use that, since $h \in\iota(E_\A^\times)$, then $h$ centralizes $\iota(E_\A^1)$ as well as $\iota(x)$. Part iii) follows from the right-invariance of the measure on $\iota(E_\A^1)\backslash B_\A^1$.

As a consequence, we obtain the following information about the group $H$ defined in \eqref{eq:stabilizerorbitalintegral}:
\begin{enumerate}
    \item [a)] We have $B^\times \subset H$. Indeed, if $h \in B^\times$ we see by part i) above that  
\begin{equation*}
\begin{aligned}
    \Lambda_{B^1}(t; \varphi^h) & = \vol(E^1 \backslash E^1_\A) \sum_{\iota \in \Emb(E, B)/B^1} \O^1(\varphi^h; E, x, \iota)\\
& = \vol(E^1 \backslash E^1_\A) \sum_{\iota \in \Emb(E, B)/B^1} \O^1(\varphi; E, x, \prescript{h^{-1}}{}\iota)\\
& = \Lambda_{B^1}(t; \varphi).
\end{aligned}
\end{equation*}
The last equality holds because $\iota \mapsto \prescript{h^{-1}}{}\iota$ is a bijection from $\Emb(E, B)/B^1$ onto itself.
    \item [b)] We have $B_\A^1 \subset H$. Indeed, this is clear from part iii) above.
    \item [c)] We have $\nrd(E_\A^\times) \subset \nrd(H)$. Indeed, by parts ii) and iii) we know that 
    \begin{equation*}
        \bigcap_{\iota \in \Emb(E, B)} B_\A^1\iota(E_\A^\times) \subset H.
    \end{equation*}
    But we have $\nrd(\iota(E_\A^\times)) = \nrd(E_\A^\times)$, which is independent of the embedding $\iota$. Since $B_\A^1$ is the kernel of $\nrd: B_\A^\times \rightarrow \A^\times$, it follows that $B_\A^1 \iota(E_\A^\times) = \nrd_{B_\A^\times}^{-1}(\nrd(E_\A^\times))$, which is independent of $\iota$. We deduce that $\nrd(E_\A^\times) \subset \nrd(H)$, as claimed. 
\end{enumerate}
Since $B^1_\A \subset H$ by part b), showing that $H = B_\A^\times$ is equivalent to showing that $\nrd(H) =\A^\times$. We know that $\nrd(B^\times) = \Q^\times$, recall Lemma \ref{lem:surjectivitynorm} above. Together with observations a) and c) this shows that 
\begin{equation*}
    \Q^\times \nrd(E_\A^\times)\nrd(M) \subset \nrd(H).
\end{equation*}
Recall that $S \subset \Pl$ is a finite set of places containing  $\infty$ such that $\widehat{\O}^{S, \times} \subset M$. Taking reduced norms it is clear that $\widehat{\Z}^{S, \times} \subset \nrd(M)$. Thus, the proof of Proposition \ref{prop:stability} is finished by applying Lemma \ref{lem:subgroupequaltoideles}.
\end{proof}

Given a positive integer $q$, we let $\Primes(q)$ denote the set of primes dividing $q$. Suppose that the function $\varphi_\A$ in the previous argument is constructed from a class function $\varphi$ on $(\O/q\O)^1$, as in equation \eqref{eq:defiadelicfunction}. In this situation we can take $S = \{\infty\} \cup \Primes(q)$, and we obtain the following consequence. 

\begin{cor}\label{cor:almoststabilityclassical}
Let $\varphi$ be a class function on $(\O/q\O)^1$. Let $t > 2$, and consider the field $E_t := \Q[X]/(X^2 - tX + 1)$. If $\Ram(E_t) \nsubseteq \Primes(q)$, then 
\begin{equation*}
    \psi_{\O^1}(t; \varphi) = \psi_{\O^1}(t; \varphi^\gamma)
\end{equation*}
for any $\gamma \in (\O/q\O)^\times$. 
\end{cor}
\begin{proof}
This follows immediately from Proposition \ref{prop:stability} and Lemma \ref{lem:adelicclassicalcomparisonB1fixedt}.
\end{proof}

\subsection{Proof of Proposition \ref{prop:almoststabilityclassical}.}\label{subsec:proofofalmoststability}
We first observe that none of the individual pieces $\psi_{\O^1}(t; \varphi)$ is too large. Recall the notation $\psi_{\O^1}(t)$ from \eqref{eq:untwistedchebyshevfixedtrace}. Recall also that $\D$ is the indefinite quaternion divison algebra which contains the maximal order $\O$.

\begin{lem}\label{lem:basicboundorbitalintegral}
For any $\varepsilon > 0$ and $t > 2$ we have 
\begin{equation*}
    \psi_{\O^1}(t) \ll_{\varepsilon, \D} t^{1 + \varepsilon}.
\end{equation*}
\end{lem}

\begin{proof}
Applying Proposition \ref{prop:classicalmatchingforstable}
we find $f$ on $\SL_2(\Z/\disc(\D)\Z)$ such that $\psi_{\O^1}(t) = \psi_{\SL_2(\Z)}(t; f)$. Thus, we have
\begin{equation*}
    \psi_{\O^1}(t) = \psi_{\SL_2(\Z)}(t; f) \ll_{\D} \psi_{\SL_2(\Z)}(t; \mathbbm{1}_{\SL_2(\Z/\disc(\D)\Z)}) = \psi_{\SL_2(\Z)}(t).
\end{equation*}
Note that 
\begin{equation*}
    \psi_{\SL_2(\Z)}(t) \leq \Psi_{\SL_2(\Z)}(N(t)) - \Psi_{\SL_2(\Z)}(N(t)- y),
\end{equation*}
for any $y > 0$. Appealing to \cite[Lemma 4]{Iwaniec1984} and recalling that $N(t) \asymp t^2$ we deduce that 
\begin{equation*}
    \psi_{\SL_2(\Z)}(t) \ll N(t)^{1/2}(\log N(t))^3 \ll_\varepsilon t^{1 + \varepsilon},
\end{equation*}
as desired.
\end{proof}

\begin{proof}[Proof of Proposition \ref{prop:almoststabilityclassical}]

Let $\varphi$ be a class function on $(\O/q\O)^1$ and let $\gamma \in (\O/q\O)^\times$. 
Let $\Primes(q)$ be the set of primes dividing $q$. For any $t > 2$, let $E_t := \Q[X]/(X^2 - tX + 1)$. By Corollary \ref{cor:almoststabilityclassical} we know that, unless  $\Ram(E_t) \subseteq \Primes(q)$, we have $\psi_{\O^1}(t; \varphi^\gamma) = \psi_{\O^1}(t;\varphi)$. Thus, we deduce that 
\begin{equation}\label{auxeq:proofofmaintheorem}
\begin{aligned}
    \Psi_{\O^1}(X; \varphi^\gamma) & = \Psi_{\O^1}(X; \varphi) + \sum_{\substack{t > 2\\
    N(t) \leq X\\ \Ram(E_t) \subseteq \Primes(q)}}\psi_{\O^1}(t; \varphi^\gamma - \varphi)\\\
    & = \Psi_{\O^1}(X; \varphi) +  O_\varphi\left(\sum_{\substack{t > 2\\
    N(t) \leq X\\ \Ram(E_t) \subseteq \Primes(q)}} \psi_{\O^1}(t)\right).
\end{aligned}
\end{equation}
Let $\E(q)$ denote a set of representatives of the isomorphism classes of real quadratic fields that ramify only at primes dividing $q$. Note that $|\E(q)| < \infty$. If $E \in \E(q)$, then by Dirichlet's unit theorem we know that
\begin{equation*}
    \o_E^1 = \{\pm \delta_E^n : n \in \Z\},
\end{equation*}  
where $\delta_E$ is a fundamental unit of $E$. Let $T_E := \{\trd(x): x \in \o_E^1\} \cap (2, \infty)$. Note that 
\begin{equation*}
  |\trd(\delta_E^n)| \asymp_E |\trd(\delta_E)|^n,
\end{equation*}
so that $T_E$ is as sparse in $(2, \infty)$ as a geometric progression. Thus, we can bound the error term in \eqref{auxeq:proofofmaintheorem} by   
\begin{equation*}
\sum_{\substack{t > 2\\
    N(t) \leq X\\ \Ram(E_t) \subseteq \Primes(q)}} \psi_{\O^1}(t) = \sum_{E \in \E(q)}\sum_{\substack{t > 2\\
    t \in T_E\\ N(t) \leq X}} \psi_{\O^1}(t) \ll_{q, \D, \varepsilon} X^{1/2 + \varepsilon},
\end{equation*}
where we have used Lemma \ref{lem:basicboundorbitalintegral}, the asymptotic $N(t) \sim t^2$, the sparsity of $T_E$ and the finiteness of $\mathcal{E}(q)$.
\end{proof}

\printbibliography

\footnotesize
\textit{Email address}: \, \texttt{alberto.reche.23@ucl.ac.uk}\par\nopagebreak
\textsc{Department of Mathematics, University College London, 25 Gordon Street, London WC1H 0AY, United Kingdom}

\end{document}